\documentclass[11pt, a4paper, reqno]{elsarticle} 
\usepackage{geometry}
\usepackage[utf8]{inputenc}
\usepackage{csquotes}
\usepackage[T1]{fontenc}
\usepackage{physics}
\usepackage[english]{babel} 
\usepackage{amsthm, amsmath, amsfonts,amssymb}
\usepackage{hyperref}
\hypersetup{
    colorlinks=true,
    citecolor=black,
    linkcolor=blue,
    filecolor=black,      
    urlcolor=blue,
    pdfauthor=Gabriele Cassese,
    pdfcreator=Gabriele Cassese,
    pdftitle=MinimalKorn
}

\newtheorem{theorem}{Theorem}[section]
\newtheorem{proposition}{Proposition}[section]
\newtheorem{corollary}{Corollary}[theorem] 
\newtheorem{lemma}[theorem]{Lemma}
\theoremstyle{definition}
\newtheorem{definition}{\textsl{Definition}}
\theoremstyle{remark}
\newtheorem{remark}{Remark}
\newtheorem*{conjecture}{Conjecture}

\makeatletter
\newtheorem*{rep@theorem}{\rep@title}
\newcommand{\newreptheorem}[2]{%
\newenvironment{rep#1}[1]{%
 \def\rep@title{#2 \ref{##1}}%
 \begin{rep@theorem}}%
 {\end{rep@theorem}}}
\makeatother

\newreptheorem{theorem}{Theorem}
\newreptheorem{lemma}{Lemma}

\begin{document}

\begin{frontmatter}
    \author{Gabriele Cassese\fnref{fn1}}
\affiliation{
    organization={Mathematical Institute, University of Oxford},
    addressline={Andrew Wiles Building},
    city={Oxford},
    postcode={OX2 6GG},
    country={United Kingdom}
    }
\fntext[fn1]{gabriele.cassese@maths.ox.ac.uk}

\date{01/12/2025}

\title{Martingales, laminates and minimal Korn inequalities}

\begin{abstract}
Korn's inequalities show that the $L^2$-norm of $\nabla u$
can be controlled by the $L^2$-norm of $\mathrm{Sym}(\nabla u)$, which only has $d(d+1)/2$ components. In [J. Math. Pures Appl. 148 (2021), pp. 199-220] Chipot
posed the question of \textit{how many scalar measurements are needed to have a 
Korn-type control on $\nabla u$} when $u$ is in $H_0^1(\Omega)$ and $H^1(\Omega)$, introducing the minimal numbers $N(d,\Omega)$ and $N'(d,\Omega)$
respectively. He proved general bounds and calculated several low-dimensional values of $N,N'$.

We reframe Chipot's problem in the language of rank-one convexity and quasiconvexity and obtain a purely algebraic characterisation of 
when such inequalities hold, which yields the sharp bounds
\begin{align*}
    N(d,\Omega)&=2d(1-o(1))\\ N'(d,\Omega)&=2d-1.
\end{align*} 
As a consequence, we recover and streamline several of Chipot's results,
we obtain a dimension-optimal Korn inequality and 
several sharp estimates for the best constant for various Korn-type inequalities. Generalisations to the rectangular case
and to general $L^p$ estimates are also considered.\par
The central new ingredient of our approach is a systematic connection
between laminates and martingales which produces explicit
families of laminates realising these bounds. This method is of 
independent interest in the calculus of variations: for instance, we use it to obtain a new quick and quantitative proof of Ornstein's
non-inequality, valid for all first order homogeneous operators in $\mathbb{R}^{2\times 2}$ and for a large
class of operators in general dimensions (including Korn's $\frac{\nabla u+\nabla u^t}2$ and $\frac{\nabla u+\nabla u^t}2-\mathrm{div}(u)\frac{\mathrm{Id}}d$).\par
\medskip
\noindent\textbf{Résumé}\par\medskip
Les inégalités de Korn montrent que la norme $L^2$ de $\nabla u$ peut être contrôlée par la norme $L^2$ de $\mathrm{Sym}(\nabla u)$, qui ne possède que $d(d+1)/2$ composantes. Dans [J. Math. Pures Appl. 148 (2021), p. 199–220], Chipot a posé la question de savoir combien de mesures scalaires sont nécessaires pour obtenir un contrôle de type Korn sur $\nabla u$ lorsque $u$ appartient à $H_0^1(\Omega)$ et à $H^1(\Omega)$, introduisant respectivement les nombres minimaux $N(d,\Omega)$ et $N'(d,\Omega)$. Il a démontré des bornes générales et calculé plusieurs valeurs de basse dimension de $N$ et $N'$.

Nous reformulons le problème de Chipot dans le langage de la convexité de rang un et de la quasiconvexité, et obtenons une caractérisation purement algébrique des cas où de telles inégalités sont valables, ce qui donne les bornes optimales
\begin{align*}
N(d,\Omega)&=2d(1-o(1))\\
N'(d,\Omega)&=2d-1.
\end{align*}
En conséquence, nous retrouvons et simplifions plusieurs résultats de Chipot, obtenons une inégalité de Korn optimale en dimension ainsi que plusieurs estimations optimales de la meilleure constante pour diverses inégalités de type Korn. Des généralisations au cas rectangulaire et à des estimations générales dans $L^p$ sont également considérées.

L’ingrédient central nouveau de notre approche est un lien systématique entre laminés et martingales, qui produit des familles explicites de laminés réalisant ces bornes. Cette méthode présente un intérêt indépendant dans le calcul des variations : par exemple, nous l’utilisons pour obtenir une nouvelle preuve rapide et quantitative de la non-inégalité d’Ornstein, valable pour tous les opérateurs homogènes du premier ordre dans $\mathbb{R}^{2\times 2}$ et pour une large classe d’opérateurs en dimensions générales, incluant les opérateurs de Korn $\frac{\nabla u+\nabla u^t}2$ et $\frac{\nabla u+\nabla u^t}2-\mathrm{div}(u)\frac{\mathrm{Id}}d$.
\end{abstract}
\end{frontmatter}

\section{Introduction}
Korn's first inequality, in its arguably most well known 
form, is the statement that for all
$u\in C^\infty_c(\mathbb{R}^d, \mathbb{R}^d),$ and all $p\in(1,\infty)$ we have 
\begin{equation*}
    \|\nabla u\|_{L^p(\mathbb{R}^d)}
\lesssim_{p,d}
    \|\mathrm{Sym}(\nabla u)\|_{L^p(\mathbb{R}^d)},
\end{equation*}
where $\mathrm{Sym}(A):=\frac{A+A^t}2$ is the symmetric 
part of $A$ and the norm considered
is the Frobenius norm for matrices, i.e. 
$\|A\|^2=\mathrm{Tr}(AA^t)$.
Korn's second inequality is more general: for $p\in (1,\infty)$ given a 
(Lipschitz) domain $\Omega$ in $\mathbb{R}^d$, for all 
$u\in W^{1,p}(\Omega, \mathbb{R}^d)$ we have
\begin{equation*}
    \|u\|_{W^{1,p}(\Omega)}
\lesssim_{p,\Omega} 
    \|u\|_{L^p(\Omega)}
+
    \|\mathrm{Sym}(\nabla u)\|_{L^p(\Omega)}.
\end{equation*}
These two inequalities have many applications
(see \cite{PoinKorn} and references therein), but these are 
not the inequalities in their strongest form.

For example, if we take 
$\mathrm{Sym}_0(A):=
\mathrm{Sym}\left(A-\mathrm{Tr}(A)\frac{\mathrm{Id}}d\right)$
(i.e., the projection of $A$ onto its trace-free symmetric 
part), we still have\footnote{Note that the second form
of this inequality does not hold for $d=2$}
\begin{equation*}
    \|\nabla u\|_{L^p(\Omega)}
\lesssim_{p,\Omega}
    \|\mathrm{Sym}_0(\nabla u)\|_{L^p(\Omega)}.
\end{equation*}
Chipot analysed this problem 
more precisely in \cite{Chipot2021OnType}: 
namely, he asked what is the smallest 
number of coordinates
that one needs on the right-hand side for an inequality
such as Korn's to hold (at least for $p=2$)?
More formally, Chipot studied the quantity $N(d),$ defined as
\begin{definition}
    Let $N(d)$ be the smallest integer $k$ for which
    there exist $\ell_1,\dots, \ell_k$ linear functionals on the space of $d\times d$ matrices $M_d(\mathbb{R})$
    such that
    \begin{equation}
    \label{eq : Ndef}
        \|\nabla u\|_{L^2(\mathbb{R}^d)}^2
    \lesssim
        \sum_{i=1}^k \|\ell_i(\nabla u)\|_{L^2(\mathbb{R}^d)}^2
    \end{equation}
    holds for all  $u\in C^\infty_c(\mathbb R^d,\mathbb{R}^d)$
    \footnote{In \cite{Chipot2021OnType}, $N$ is defined 
    for $u\in C^\infty_c(\Omega, \mathbb{R}^d)$, so a priori 
    $N$ depends on the domain. 
    However, one easily sees that $N$
    is independent of the domain, so we skip that here.}.
\end{definition}
Similarly, he defined $N'(d)$ as
\begin{definition}
    Let $\Omega$ be an open subset of $\mathbb{R}^d$. Let $N'(d, \Omega)$ be the smallest integer $k$ for which
    there exist $\ell_1,\dots, \ell_k$ linear functionals on $M_d(\mathbb{R})$
    such that
    \begin{equation}
    \label{eq : N1def}
        \| u\|_{H^1(\Omega)}^2
    \lesssim
        \|u\|_{L^2(\Omega)}^2
    +
        \sum_{i=1}^k \|\ell_i(\nabla u)\|_{L^2(\Omega)}^2
    \end{equation}
    holds for all $u\in C^\infty(\Omega,\mathbb{R}^d)$
    \footnote
    {
        More care is required
        in handling $\Omega$ in this case. Indeed, the 
        regularity of $\Omega$ seems to play a more 
        important role, as it is well known
        that there exist open sets 
        $\Omega$ on which Korn's second
        inequality fails
        (see \cite{JohnKorn}).
    }.
\end{definition}
As we will see, $N'(d,\Omega)$ does not depend on the domain provided the domain is regular enough (say, Lipschitz), so we will continue using $N'(d)$
to denote $N'(d,\Omega)$ for any bounded Lipschitz domain. Henceforth, we will
also assume all domains to be bounded and Lipschitz.
Chipot proved many important properties 
of $N(d)$ and $N'(d),$ which we briefly 
remind the reader:
\begin{theorem}[{\cite{Chipot2021OnType}}] For any $d,k\in\mathbb{N}$ we have
\begin{enumerate}
    \item (Proposition 2.1) $N(d)\le \frac{d(d-1)}2+1$
    \item (Theorem 2.3) $N(d)\ge d$ and the inequality
        is strict if $d$ is odd.
    \item (Theorem 2.4, Theorem 4.1) $N(2)=2, N(3)=4, N(4)=4,  N(8)\le 16$
    \item (Theorem 3.3) $N(kd)\le k^2 N(d)$
    \item (Theorem 5.1) For any bounded Lipschitz domain $\Omega\subset \mathbb{R}^2$, $N'(2, \Omega)=3.$ This in particular implies that $N'(2)$ is well defined.
    \item (Theorem 5.2) For any bounded Lipschitz domain $\Omega$, $N'(d,\Omega)\ge N(d).$
\end{enumerate}
\end{theorem}
In particular, the natural question of the
growth of $N(d)$ remained open: does $N(d)$ grow quadratically, or linearly,
or at a different rate?
It is not hard to 
see that one can 
without loss of generality
assume the functionals to 
be orthogonal to each other, so the question is 
essentially how many coordinates of $\nabla u$ are needed to define
the norm on $W^{1,p}$.
In particular, if $N(d)\approx d$ holds, 
this implies that, as $d$
grows, only a vanishingly small number of coordinates
of $\nabla u$ are really
needed to determine a norm,
i.e. the standard gradient norm is quite redundant.
In this regard, $N'(2)\neq N(2)$ is rather surprising,
as there seems to be no
intrinsic difference between the two spaces.

Our results hinge on the following observation (various forms of this 
claim have been known in the community of calculus of variations
for a long time; see remark \ref{rmk : hist} for background and history of this result):
\begin{theorem}\label{thm : solchipot}
Given a field $\mathbb K$, let $n_{\mathbb K}(d,k)$ denote the maximal dimension
a subspace $Y$ of $M_d(\mathbb K)$ can have without intersecting the set $\mathcal R_k$ of non-zero matrices
of rank at most $k$. Then
    \begin{equation}
        N(d)=d^2-n_{\mathbb R}(d,1)
    \end{equation}
    \begin{equation*}
        N'(d)=d^2-n_{\mathbb{C}}(d,1)
    \end{equation*}
\end{theorem}
This connection with algebra allows us to attack the problem from a very different angle and this will lead us to prove several interesting properties of $N(d), N'(d),$
by taking advantage 
of what is known about $n_{\mathbb{R}}(d,1)$.
We mention here the most important ones. The first, and perhaps most surprising, results
we obtain are
\begin{corollary}\label{cor : asymptsolchip}
For each $d\in \mathbb N$ larger than one we have
    \begin{equation*}
        N(d)\le 2d-2.
    \end{equation*}
    This bound is attained if $d=2^n+1$ for some $n\in\mathbb N$.
    Moreover, asymptotically we have $N(d)\sim 2d$, so the bound is asymptotically sharp.
\end{corollary}
\begin{corollary}
For each $d\in \mathbb N$ we have
    \begin{equation*}
        N'(d)=2d-1.
    \end{equation*}
    In particular, the inequality $N'(d)\ge N(d)$ is always strict when $d>1$.
\end{corollary}
We also manage to improve on the best previously known lower bound:
\begin{corollary}
Equality $N(d)= d$ is attained if and only if $d\in\{1,2,4,8\}.$ Moreover, we have
\begin{equation*}
    N(d)\ge 2^{\lceil\log_2(d)\rceil}.
\end{equation*}
\end{corollary}
These corollaries by themselves are enough to calculate $N(d)$ for $d$ up to $9$; we will
later see, using more convoluted methods, how to calculate $N(d)$ for $d$ up to $19$
(and we will provide lower and upper bounds for $N(d)$ up to $33$).
For the
sake of the exposition, and as most of the proofs of the 
required properties of $n_{\mathbb K}$ and related
functions are rather elementary to prove, we will 
include their proofs where possible,
trying to make them more 
accessible
to analysts (which amounts to  keeping the algebraic topology tools to a 
minimum). Doing so will also allow us to obtain streamlined 
and shorter proofs of Chipot's original results.\par
In order to estimate the constants present in these inequalities, we 
develop a systematic connection between martingales and laminates which allows us
to bound the constant from below and (for rank-one convexity) from above
by explicitly building laminates that witness these bounds.
\par
As a by-product of our study, we will obtain an explicit 
and much strengthened form of Korn's first and second 
inequalities which, to our knowledge, has not been explicitly noted 
before.
Namely,  we will prove
\begin{equation*}
    \|u\|_{H^1(\Omega)}
\lesssim
    \|u\|_{L^2(\Omega)}
+
    \|P_{\mathcal H}(\nabla u)\|_{L^2(\Omega)}, 
\end{equation*}
where $\mathcal H$ is the space of Hankel matrices,
i.e. matrices that are constant on each skew-diagonal, and $P_{\mathcal H}$
denotes the orthogonal projection
from $M_d(\mathbb{R})$ onto $\mathcal H$.
Note that $\mathcal H\subset \mathrm{Sym}$ and 
$\mathrm{dim}(\mathcal H)=2d-1,$ so this is indeed a 
strengthening of Korn inequalities (and thanks to $N'(d)=2d-
1$, this is dimensionally optimal). The results mentioned so far
are all consequences of the previously mentioned algebraic connection (and will be explored in sections \ref{sec : Chipot}, \ref{sec : second}
and \ref{sec : rect}, respectively concerning $N(d), N'(d)$ and the rectangular case);
in order to estimate the constants involved, a more refined set
of tools is needed. With this goal, we discuss in section \ref{sec : calcvar} a connection between laminates and martingales which we then exploit
to calculate the constants. This allows us to prove a bound for the first Korn-Hankel inequality which is sharp up to $\mathcal O(\log(d))$, and several new bounds for $p\neq 2$. This does not exhaust the potential
applications of the tools we introduce and as an example of this we 
provide, in section \ref{sec : Ornstein}, a new quantitative proof 
of Ornstein's non-inequality result.
\subsection{Notation}
In this paper, we use the standard Vinogradov notation where,
given two positive quantities $A,B$,
$A\lesssim B$ means that there exists a positive constant 
$C$ such that $A\le BC$. 
We write $A\lesssim_x B$ if the implicit constant $C$ is allowed to depend on $x$. 
Given two quantities $A$ and $B$, $A\approx B$ means $A \lesssim B$
and $A\gtrsim B$.
We will use $\mathbb K$ to denote an unspecified
field.
For $\mathbb K=\mathbb{R},\mathbb{C},$ we will always consider 
$M_{m,d}(\mathbb K)$ (the space of $m\times d$ matrices 
on $\mathbb K$) to be equipped with the Frobenius inner 
product, that is $\langle A,B\rangle=\mathrm{Tr}(AB^*)$ and associated Frobenius norm $\|A\|^2=\mathrm{Tr}(AA^*)$; when $m=d,$ we will write $M_{d}(\mathbb K)$.
Given a closed subspace $X$ of a Hilbert space $H$, we will write $P_X$
for the orthogonal projection on $X$ and $Q_X$ for the orthogonal projection on $X^\perp$. We use standard notation for functions and
function spaces such as can be found, for instance, in \cite{Dacorogna2004IntroductionVariations}.
\subsection{A roadmap}
The paper has two main threads. First, Chipot's work left open the
growth of $N(d)$ and the higher-dimensional behaviour of $N'(d)$;
we resolve these questions by proving the algebraic formulae
$N(d)=d^2-n_{\mathbb R}(d,1)$ and
$N'(d)=d^2-n_{\mathbb C}(d,1)$, from which the sharp bound
$N'(d)=2d-1$ and the asymptotically sharp estimate
$N(d)=2d(1-o(1))$ follow. The route to these formulae is:
Korn-type inequalities are first translated into rank-one convexity
and quasiconvexity conditions, these conditions are then reduced to
the algebraic problem of finding large subspaces of matrices avoiding
rank-one directions, and the resulting dimension estimates are finally
obtained from classical results on nonsingular bilinear maps and
projective spaces. This is why algebraic topology appears in what is
initially a PDE question. The second thread is quantitative: a
connection between laminates and martingales is introduced so that the
rank-one convex obstructions encoded by laminates can be combined with
Burkholder's sharp martingale inequalities, yielding constants and
sharpness information that the qualitative laminate picture alone does
not provide (at least, not as easily).

\section{Preliminary results in calculus of variations}\label{sec : calcvar}
We briefly recall some concepts from
calculus of variations:
a function
$f\colon M_{m,d}(\mathbb{R})\to \mathbb{R}$ is said
to be rank-one convex if it is 
convex on each 
rank-one line, i.e. on each 
segment connecting two matrices $A,B$
such that
$\mathrm{rank}(A-B)=1$. A 
rank-one convex function 
is locally
Lipschitz (\cite[lemma 5.6]{Rindler}).
Similarly, a continuous 
function $f$
is said to be quasiconvex if, for all
$u\in W^{1,\infty}_0(B_{\mathbb{R}^d}(0,1),\mathbb{R}^m)$
we have
\begin{equation*}
    \int_{B_{\mathbb{R}^d}(0,1)} f(A+\nabla u)\ge |B_{\mathbb{R}^d}(0,1)| f(A),
\end{equation*}
where $|B_{\mathbb R^d}(0,1)|$ denotes the $d$-dimensional Lebesgue
measure of the unit ball $B_{\mathbb{R}^d}(0,1)$.
It is well known (see \cite[Proposition 5.3]{Rindler} and \cite[Lemma 4.3]{Muller1999VariationalTransitions})
that quasiconvex functions are rank-one convex, and the opposite implication (known
as Morrey's problem) is known to be false for $m\ge 3, d\ge 2$, as proved by Šverák (\cite{Sverak1992Rank-oneQuasiconvexity}, see also \cite{Grabovsky2017}). In dimension $m=2, d\ge 2$, it is not known whether rank-one convexity 
and quasiconvexity are equivalent in general, though several positive results are known for 
particular functions and particular domains (see \cite{Harris2018}, \cite{Faraco2008}, \cite{Muller1999Rank-oneMatrices}).
We define the
rank-one convex envelope $f^{rc}$ and the quasiconvex envelope $f^{qc}$ 
of a function $f\colon M_{m,d}\to \mathbb{R}$ respectively
as
\begin{equation*}
    f^{rc}(A)=\sup\{g(A)\colon  g\le f\ \mathrm{and}\ g\ \mathrm{ rank-one \ convex}\}
\end{equation*}
and
\begin{equation*}
    f^{qc}(A):=\sup\{g(A)\colon  g\le f\ \mathrm{and}\ g\ \mathrm{ quasiconvex}\}.
\end{equation*}
The envelopes admit an alternate description by duality. For $f^{rc}$ we have 
(see 
\cite[Theorem 6.10]{Dacorogna2008DirectVariations}
for the formula for the envelope,
\cite{Pedregal1993LaminatesMicrostructure} and \cite[Section 9.1]{Rindler} for
definitions and generalities on
laminates):
\begin{equation}\label{eq : rcenv}
    f^{rc}(A)=\inf \left\{\langle f,\mu\rangle\colon  \mu\in\mathcal M_A\right\},
\end{equation}
where $\mathcal M_A$ represents
the set of laminate measures (of finite order) having
$A$ as their barycentre.
Similarly, for the quasiconvex envelope
we have Dacorogna's formula (\cite[Theorem 6.9]{Dacorogna2008DirectVariations}, \cite[Appendix]{Kinderlehrer1991CharacterizationsGradients})
\begin{equation}
\label{eq : qcenv}
    f^{qc}(A)
=
    \underset{
    u\in W^{1,\infty}_0(B_{\mathbb{R}^d}(0,1), \mathbb{R}^m)
    }
\inf
    \frac{1}{|B_{\mathbb{R}^d}(0,1)|}
    \int_{B_{\mathbb{R}^d}(0,1)}f(A+\nabla u(x))\mathrm{d} x
\end{equation}
It is not hard to prove that, for
a continuous function $f$, either
$f^{qc}\equiv-\infty$ or 
$f^{qc}>-\infty$ and 
$f^{qc}$ is quasiconvex, 
and similarly for rank-one convexity.
It follows that for homogeneous functions $f$, the existence of a lower quasiconvex bound 
is equivalent to $f^{\mathrm{qc}}(0)=0$, a necessary condition for which is that $f^{\mathrm{rc}}(0)=0$.\par
In this section, we focus on determining when the
rank-one convex envelope of certain functions is real valued (in other words, 
whether $f$ has a lower rank-one convex envelope). The functions 
of interest have the form
\begin{equation}\label{eq : shape}
	f_{X,p,C}(A)
=
	C^p\| P_X(A)\|^p-\| Q_X(A)\|^p,
\end{equation}
where $X$ is a subspace of $M_d(\mathbb{R})$, $\|\cdot\|$ is the Frobenius norm
and $P_X, Q_X$ are the associated
orthogonal projections
onto $X$ and $X^\perp$
respectively.
 In the following we will
denote a function of the class above as $f$ or $f_p$
when $X$ and $C$
are clear from the context.\par
The reasons for studying this class of functions
are manifold. In this paper, we are interested in their 
relations with differential inequalities.
In this regard, we note that
$f^{qc}$ is real valued iff
$f^{qc}(0)=0$, and that this is 
equivalent, thanks to \eqref{eq : qcenv},
to the inequality
\begin{equation}\label{eq: coerc}
	\| Q_X(\nabla u)\|_{L^p(\mathbb{R}^d)}
\le
	C \| P_X(\nabla u)\|_{L^p(\mathbb{R}^d)}
\end{equation}
holding for all $u\in C^\infty_c(\mathbb{R}^d,\mathbb{R}^d)$.
Since rank-one convexity is a necessary
condition for quasi-convexity, such a study is naturally
interesting.\par
For $p=2$, the function above becomes a quadratic form, namely
\begin{equation*}
	f(A)
=
	C^2\langle P_X(A),A\rangle
-
	\langle Q_X(A), A\rangle.
\end{equation*}
For quadratic functions, it is known (and easy to prove, see
\cite[Theorem 5.25(i)]{Dacorogna2008DirectVariations}) that
quasiconvexity is equivalent to rank-one convexity which in turn is 
equivalent to $f(a\otimes b)\ge 0$ for all $a,b\in\mathbb{R}^d$. Hence, 
$f^{qc}(0)=0$ is equivalent to
\begin{equation*}
	C^2\langle P_X(A),A\rangle
-
	\langle Q_X(A), A\rangle
=
	f(A)
\ge
	 0
\end{equation*}
for all $A$ such that $\mathrm{rank}(A)=1$. In other words,
\begin{equation*}
	\inf_{A\, \colon \, \mathrm{rank}(A)=1}
	\frac{\|P_X(A)\|^2}{\|Q_X(A)\|^2}
\ge
	\frac1{C^2}.
\end{equation*}
In this section, we want to extend this result to $p\in (1,\infty)$
and provide a sharpness result there as well.
In particular, we will prove the following:
\begin{theorem}\label{thm : mainthm}
Let $f_C$ be a function of the form \eqref{eq : shape}.
Then there exists a $C$ such that $f_C^{rc}$ is real-valued
if and only if $X^\perp$ does not contain rank-one matrices.
In that case, we can choose $C=(p^*-1)B$,
where 
\begin{equation*}
	B^{-1}
=
	\inf_{A\colon \mathrm{rank}(A)=1}
	\frac{\|P_X(A)\|}{\|Q_X(A)\|}
\end{equation*}
and $p^*=\max\left(p,\frac{p}{p-1}\right)$.
\end{theorem}
In some particular cases,
we can also obtain a sharp estimate on $C$:
\begin{theorem}\label{thm : sharp}
Let $X$ be a subspace of $M_d(\mathbb{R})$
such that 
\begin{enumerate}
\item For all matrices $R\in M_d(\mathbb{R})$ having rank-one, 
\begin{equation*}
	\|P_X(R)\|
\ge
	\frac1c\|Q_X(R)\|.
\end{equation*}
\item There exist  $A\in X, B\in X^\perp$ such that $\|A\|=\|B\|=1$,
$\mathrm{rank}(A\pm cB)=1$.
\end{enumerate}
Then the function
\begin{equation*}
	f(A)
=
	C^p\| P_X(A)\|^p
-
	\| Q_X(A)\|^p
\end{equation*}
has real-valued rank-one convex envelope
if and only if $C\ge c(p^*-1).$
\end{theorem}
To prove this result, we will need to introduce a connection between laminates and martingales.
Before we do so formally, let us describe the connection heuristically:
a finite order laminate (with barycentre in $0$) is constructed by applying a series of splits
to the original Dirac measures $\delta_0$, that is to say operations 
where one passes from the measure $\delta_A$
to $t \delta_{B}+(1-t)\delta_C$ (for some appropriate $t,B,C$) which preserve the barycentre of the measure. This can be seen 
as a martingale on the dyadic tree, with each split represented by a branching in the tree.
Let us set some notation concerning dyadic trees:
\begin{definition}[Dyadic trees]
    Let us denote the dyadic tree by $\mathcal T$,
i.e. the set of finite sequences
$(\varepsilon_0,\dots, \varepsilon_n)$ 
with elements
$\varepsilon_{i}\in\{0, 1\}$.
Given such a sequence $\sigma, |\sigma|$ will denote its length $|\sigma|=n+1$.
We will write $T_m:=\{\sigma\in\mathcal T\colon |\sigma|=m\},
\mathcal T_m:=\{\sigma\in\mathcal T\colon |\sigma|\le m\}$.
Given a sequence $\sigma=(\epsilon_0, \epsilon_1,\dots,\epsilon_m)$ and a number $\varepsilon\in\{0,1\}$, we will 
write $\sigma \frown \varepsilon$ for the sequence obtained by adding $\varepsilon$
in final position, i.e.
\begin{equation*}
    \sigma\frown \varepsilon=(\epsilon_0,\dots, \epsilon_m,\varepsilon).
\end{equation*}
\end{definition}
\begin{definition}[Associated martingale]
Let $\mu$ be a laminate of order $n$.

Then there exists a function
$M\colon \mathcal T_n\to M_d(\mathbb{R})$ such that for all $x\in \mathcal T_{n-1}$
\begin{equation}
	t(x)M(x\frown  0)
+
	(1-t(x))M(x\frown  1)
=
	M(x),    
\end{equation}
and
\begin{equation}\label{eq : rankmart}
	\mathrm{rank}\big(
		M(x\frown  0)
	-
		M(x\frown  1)
		\big)
\le
	1
\end{equation}
Let $(X_i)_{i\in \mathbb{N}\cup\{0\}}$
denote the random walk on $\mathcal T$
such that for every $x\in \mathcal T_n$,
whose length we denote as $m$, we have
\begin{equation*}
    \mathbb P(X_{m+1}=x\frown 
0|X_m=x)=t(x)
\end{equation*}
and
\begin{equation*}
    \mathbb P(X_{m+1}=x\frown 1|X_m=x)=1-t(x).
\end{equation*}
We define the 
associated dyadic martingale
as $M(X_{i\wedge n})$.
\end{definition}
Condition \eqref{eq : rankmart} can be
reformulated in a more natural manner:
if we denote by $\mathrm{d}M_i$ the increments
of the martingale (i.e. $\mathrm{d}M_i=M_{i+1}-M_i$), then assuming $X_{i+1}=x\frown 0$
we have
\begin{align*}
    \mathrm{d}M_i&=M(x\frown 0)-M(x)
    \\&=M(x\frown 0)-t(x)M(x\frown 0)-(1-t(x))M(x\frown 1)\\&=(1-t(x))(M(x\frown 0)-M(x\frown 1))
\end{align*}
and similarly if we instead assume $X_{i+1}=x\frown 1$
\begin{equation*}
    \mathrm{d}M_i=-t(x)(M(x\frown 0)-M(x\frown 1)),
\end{equation*}
so condition \eqref{eq : rankmart} is
satisfied if and only if $\mathrm{rank}(\mathrm{d}M_i)\le 1$ almost everywhere.
\begin{remark}
    Many other results connecting certain classes of martingales
    with convexity properties are known. For example, Aumann and Hart
    in \cite{Aumann1986Bi-convexityBi-martingales} found a similar result for
    bi-convexity and bi-martingales.
\end{remark}
We will also need the following result concerning martingales:
\begin{theorem}[Burkholder \cite{Burkholder}, see \cite{osekowski} for an exposition]\label{thm : burk}
Let $\mathcal H$ be a Hilbert space and let
$X_i, Y_i$ be two $\mathcal H$-valued martingales
(with respect to the same filtration). 
If $(X_i)$ is differentially subordinated to $(Y_i)$,
i.e. $\|\mathrm{d}X_{i}\|\le \|\mathrm{d}Y_{i}\|$ for all $i\in\mathbb N$ and $\|X_0\|\le \|Y_0\|$, then for all $i\in \mathbb{N}$
\begin{equation}\label{eq : burk}
	\|X_i\|_{L^p(\mathbb P)}
\le
	(p^*-1)\|Y_i\|_{L^p(\mathbb P)},
\end{equation}
where $p^*=\max\left(p,\frac{p}{p-1}\right)$.
Moreover, the inequality is sharp
\footnote{
	That is to say, no smaller constant works
	}
among martingales defined on dyadic trees such that $\|\mathrm{d}X_i\|=\|\mathrm{d}Y_{i}\|.$
\end{theorem}
\begin{proof}[Proof of Theorem \ref{thm : mainthm}]
It suffices to prove that $f_{X,p,C_p}^{rc}(0)=0$.
To do so, let us recall that
\begin{equation*}
	f^{rc}(x)
=
	\inf_{\mu\colon  [\mu]=x, \mu\text{ is a laminate}}
	\langle f,\mu\rangle.
\end{equation*}
So it suffices to prove that $\langle f_{X,p,C_p},\mu\rangle\ge 0$
for laminates with barycentre $0$ if $C\gtrsim (p^*-1)$.
Let $\mu$ denote such a laminate, and let $n$ be its order.
By hypothesis, we have
\begin{equation*}
	\sup_{\mathrm{rank}(A)=1}
	\frac{\|Q_X(A)\|}{\|P_X(A)\|}
=
	k
<
	\infty.
\end{equation*}
We claim that $C\ge k(p^*-1)$ suffices.
For simplicity, we assume $k=1$,
the proof in the general case is the same mutatis mutandis.
We now associate to $\mu$ an $M_d(\mathbb{R})$-valued martingale
as in the definition above, which we call $M_i$. 
It is easy to see that
$A_i:= P_{X}(M_i),
B_i:= Q_{X}(M_i)$
are martingales and that $A_0=B_0=M_0=0.$
Moreover, since $\mathrm{rank}(M_{i+1}-M_i)\le 1$ and 
$\|P_{X}(R)\|\ge \|Q_{X}(R)\|$
for all rank-one matrices,
it follows that $B_i$ is differentially subordinated to $A_i$.
Using Burkholder's inequality \eqref{eq : burk},
it follows that if $c_p\ge (p^*-1),$ we have
\begin{equation*}
	c_p\|A_i\|_{L^p(\mathbb P)}
\ge
	\|B_i\|_{L^p(\mathbb P)}
\end{equation*}
for all $i\in \{0,\dots, n\}$ and taking $i=n$ we get  that
   \begin{equation*}
	\langle f,\nu\rangle
=
	\int f\mathrm{d} L_{M_n}
=
	c_p^p\|A_n\|^p_{L^p(\mathrm{d} \mathbb P)}
-
	\|B_n\|^p_{L^p(\mathrm{d} \mathbb P)}
\ge
	0,
\end{equation*} 
where $L_{M_n}$ is the law of the random
variable.
On the other hand, assume that $X^\perp$ contains a rank-one matrix $R$, 
which we assume without loss of generality to have norm $1$.
Then
\begin{equation*}
	f_{X,p,C}^{rc}(0)
\le
	\inf_{t>0}
	\frac{f_{X,p,C}(-tR)
+
	f_{X,p,C}(tR)}{2}
=
	\inf_{t>0}-t^p
=
	-\infty.
\end{equation*}
\end{proof}
\begin{proof}[Proof of Theorem \ref{thm : sharp}]
To prove the other direction,
i.e. $C<c(p^*-1)$ implies $f_{p,c_p}^{rc}(0)=-\infty$,
we focus on a particular class of martingales:
namely, let $f_n$ be any dyadic martingale and $g_n$
a $\pm1$-transform of $f_n$,
i.e. a martingale such that $dg_n=\varepsilon_n df_n$ with $
\varepsilon_n$ a predictable sequence of signs.
Then, we define
\begin{equation*}
	M_n
=
	A f_n
+
	cBg_n.
\end{equation*}
This matrix-valued martingale satisfies $\mathrm{rank}(dM_n)\le 1$ a.e.,
so it is associated to a laminate with barycentre at $0$.
From the sharpness statement in Theorem \ref{thm : burk}
the result follows.
\end{proof}
\begin{remark}
       The problem of determining conditions for the rank-one convexity
    of a function (in particular as a necessary condition for quasiconvexity)
    is a very active area of research with a long history;
 see \cite{Banuelos}, \cite{BaernsteinII1997}, \cite{Iwaniec2002Nonlinear},
 \cite{Astala2012BurkholderMappings} and references therein for a more detailed history of this approach
    in its relation with the Iwaniec-Martin conjecture. The technique of using laminates has
    also been in use, in various forms, for quite some time, see
    \cite{Kinderlehrer1991CharacterizationsGradients}. 
    We particularly highlight Faraco's results with the use of the so-called staircase
    laminates (see \cite{Faraco2003}) and
    \cite{Boros2013}, where a connection
    is noted between laminates and the Burkholder function in the context of studying
    the rank-one convexity of some functions related to the
    Iwaniec-Martin conjecture, which is in a similar
    spirit to ours. Theorem \ref{thm : sharp} was 
    originally developed in a rougher form
     in \cite{Cassese},
    where we applied it to a particular function.
    We believe that the results we have here do not exhaust
    the applicability of the technique, which will be broadened in future work. 
\end{remark}
\section{Korn's first inequality and the quantity \texorpdfstring{$N(d)$}{N(d)}}\label{sec : Chipot}
The goal of this section is to prove the first part of Theorem \ref{thm : solchipot},
which we restate here for the reader's convenience:
\begin{reptheorem}{thm : solchipot}[Part $1$]
\begin{equation}
	d^2
-
	n_{\mathbb R}(d,1)
=
	N(d).
\end{equation}
\end{reptheorem}
From this, we will deduce a sharp lower bound, an asymptotically 
sharp upper bound, we will calculate all 
$N(d)$ up to $19$ and provide ranges for $N(d)$ up to $d=33$ 
(though these results could quite likely be extended further, we
refrain from doing that here).
Before proving the theorem, we note that the
definition can be restated in a way that is slightly
more amenable to our process:
\begin{lemma}
Let $M(d)$ be defined as
\begin{equation*}
	M(d)
:=
	\inf 
	\left\{
		\mathrm{dim}(X)\colon  \forall u\in C^\infty_c(\mathbb{R}^d,\mathbb{R}^d)\  
		\|\nabla u\|_{L^2(\mathbb{R}^d)}
	\lesssim
		\| P_X(\nabla u)\|_{L^2(\mathbb{R}^d)}
	\right\}.
\end{equation*}
Then $M(d)=N(d)$.
\end{lemma}
\begin{proof}
    First, let $X$ be a subspace of $M_d(\mathbb{R})$ such that
    $\|\nabla u\|_{L^2(\mathbb{R}^d)}\lesssim \| P_X(\nabla u)\|_{L^2(\mathbb{R}^d)}$
    for all smooth compactly supported functions. Let
    $e_1,\dots, e_{\mathrm{dim}(X)}$ be an orthogonal 
    basis of $X$. Then let $l_i(\cdot):=\langle e_i, \cdot\rangle$.
By construction we have
\begin{equation*}
	\|\nabla u\|_{L^2(\mathbb{R}^d)}^2
\lesssim
	\| P_X(\nabla u)\|_{L^2(\mathbb{R}^d)}^2
=
	\sum \|l_i(\nabla u)\|_{L^2(\mathbb{R}^d)}^2.
    \end{equation*}
It follows that $N(d)\le M(d).$ To prove the other direction,
let $l_1,\dots, l_m$ be a set of linear functionals
witnessing $N(d)=m.$
Select for each of them a vector $x_i\in M_d(\mathbb{R})$
such that $l_i(\cdot)=\langle x_i,\cdot\rangle$.
Define $X:=\text{span}(x_1,\dots, x_m)$. By minimality
of the set $\{l_i\}$, $\{x_1,\dots, x_m\}$
is a basis of $X$. 
By finite dimensionality of $X$ it then follows that
\begin{equation*}
	\| P_X(y)\|^2
\approx_{\mathrm{dim}(X)}
	\sum |\langle x_i, y\rangle|^2,
\end{equation*}
hence
\begin{equation*}
	\|\nabla u\|_{L^2(\mathbb{R}^d)}^2
\le
	\sum \|l_i(\nabla u)\|_{L^2(\mathbb{R}^d)}^2
=
	\sum \|\langle x_i, \nabla u\rangle\|_{L^2(\mathbb{R}^d)}^2
\approx_{\mathrm{dim}(X)}
	\| P_X(\nabla u)\|_{L^2(\mathbb{R}^d)}^2,
\end{equation*}
proving the other direction $N(d)\ge M(d).$
\end{proof}
\begin{proof}[Proof of Part $1$ of Theorem \ref{thm : solchipot}]
The definition can be equivalently  stated as:
let $\mathcal X$ be the class of subspaces $X$ of $M_d(\mathbb{R})$
for which there exists a constant $C=C_X$ such that
\begin{equation*}
 C\|P_X(\nabla u)\|_{L^2(\mathbb{R}^d)}^2\ge \|Q_X(\nabla u)\|_{L^2(\mathbb{R}^d)}^2   
\end{equation*}
holds
for all admissible $u$. In these
terms we have $N(d)=\min\{\mathrm{dim}(X)\colon X\in\mathcal X\}$.
Let us define $f_X\colon M_d(\mathbb{R})\to \mathbb{R}$ as
\begin{equation*}
	f_{C,X}(A)
:=
	C\|P_X(A)\|^2
-
	\|Q_X(A)\|^2.
\end{equation*}
Hence $X$ is admissible if and only
if there exists a $C$ such that
\begin{equation*}
	f_{C,X}^{qc}(0)
=
	0.
\end{equation*}
Since $f_{C,X}$ is a quadratic form, 
this is equivalent to 
\begin{equation*}
	f_{C,X}^{rc}(0)
=
	0.
\end{equation*}
By applying Theorem \ref{thm : mainthm}, 
one obtains that such  $C$ exists if and only if $X^\perp$
avoids rank-one matrices, proving the result.
\end{proof}
\begin{remark}\label{rmk : hist}
    We are of course not the first to observe
    a connection between rank-one
    properties and inequalities such 
    as Korn's. Indeed, results of this kind have been known
    for a long time, see for instance \cite{Chen2017OnIntegrals},
    \cite[Lemma 2.7]{Muller1999VariationalTransitions},
    \cite{Conti2005AFunctions}, \cite{Faraco2022}, \cite{KirchheimJan}.
    Similar
    results also hold in other spaces: Boman \cite{Boman1967PartialSpaces} 
    investigated
    a problem with a method that
    is reminiscent of our
    approach (or, to be 
    more precise, vice versa): in particular, his condition
    $\hat A$ is equivalent to 
    $\mathrm{span}(\hat A)$ being,
    in our vocabulary, admissible. In 
    this light, our result can be seen
    as a $W^{1,p}$ version of 
    his Theorem 1. \par
    It seems however that
    this algebraic characterisation
    has never been taken full advantage of as we do here.\end{remark}
\begin{remark}
An alternative proof can easily be obtained by 
using ellipticity of the operator $T_X(u):=P_X(\nabla u)$
(and indeed it is not hard to see that 
Theorem 3.2 in \cite{Chipot2021OnType}
is essentially the ellipticity condition for $T_X$).
Our proof, however, has the advantage of allowing one to
control the constants associated with the inequalities,
whereas using ellipticity only ensures the existence of such
constants.
\end{remark}
The number $n(d,1)$ can also be interpreted in a different
way, which will prove itself to be more amenable to our study.
\begin{definition}
Let $d,k\in \mathbb N$, and let 
$\mathcal B(\mathbb K^d\times \mathbb K^d, \mathbb K^k)$
denote the set of bilinear maps $f\colon\mathbb K^d\times \mathbb K^d\to \mathbb K^k$.
We say  that such a map is nonsingular
if $f(a,b)=0$ implies that 
$a=0$ or $b=0$.
If $\mathcal B_{ns}(\mathbb{R}^d\times 
\mathbb{R}^d,\mathbb R^k)$ denotes the
set of nonsingular bilinear mappings, 
we define $g(d)$ as
\begin{equation*}
	g_{\mathbb K}(d)
:=
	\min
	\left\{
		k\colon 
		\mathcal B_{ns}(\mathbb K^d\times \mathbb K^d, \mathbb K^k)
	\neq
		\emptyset
	\right\}.
\end{equation*}
\end{definition}
\begin{proposition}
For each $d\in \mathbb{N}$ we have
\begin{equation*}
	d^2
-
	n_{\mathbb K}(d,1)
=
	g_{\mathbb K}(d).
    \end{equation*}
\end{proposition}
\begin{proof}[Proof (folklore)]
Note that $f\colon\mathbb K^d\times \mathbb K^d\to \mathbb K^k$
with $k$ being minimal implies $f$ is surjective.
Moreover, note that any such map $f$ bilinear and nonsingular
naturally induces
$L\colon \mathbb K^d\otimes \mathbb K^d\to \mathbb K^{k}$,
so taking $L^{-1}(0)$ provides a subspace
avoiding rank-one matrices, and 
\begin{equation*}
	\mathrm{dim}(M_d(\mathbb K))
-
	\mathrm{dim}(L^{-1}(0))
=
	\mathrm{dim}(\mathbb K^k)
=
	g_{\mathbb K}(d),
\end{equation*}
so 
\begin{equation*}
	d^2
-
	n_{\mathbb K}(d,1)
\le
	g_{\mathbb K}(d).
\end{equation*}
On the other hand given such a dimension-minimal subspace $X$,
one can define $L$ as the natural quotient map
\begin{equation*}
	L\colon 
	\mathbb K^d\otimes \mathbb K^d
\simeq
	M_d(\mathbb K)
\twoheadrightarrow
	M_d(\mathbb K)/X
\simeq
	\mathbb K^{d^2-n_{\mathbb K}(d,1)}
\end{equation*} 
and this implies
\begin{equation*}
	g_{\mathbb K}(d)
\le
	d^2
-
	n_{\mathbb K}(d,1).
\end{equation*}
 \end{proof}
 As we will mostly focus on $g_{\mathbb{R}}(d)$, we will from now on
 write $g(d)$ for $g_{\mathbb R}(d)$.
 \subsection{First consequences}
The function $g(d)$ has been studied in depth
since the $1960$s.
In particular, as we will see, there exist asymptotically
sharp lower and upper bounds for it and we can now
transfer those to $N(d)$. Since some
of the results for $g(d)$ are
rather complex and usually derived with 
techniques from algebraic topology, we first focus on
using analytical techniques.
To start, let us show how one can use this result
to obtain short proofs of some of the results obtained
in \cite{Chipot2021OnType}.
\begin{corollary} $N(d)$ is non-decreasing.
\end{corollary} 
\begin{proof}
The corollary follows from composing the nonsingular bilinear map 
$f\colon\mathbb{R}^{d+1}\times \mathbb{R}^{d+1}\to \mathbb{R}^{N(d+1)}$
with the standard inclusion
$\iota\colon \mathbb{R}^d\times \mathbb{R}^d\to \mathbb{R}^{d+1}\times \mathbb{R}^{d+1}$
defined as $\iota(x,y)=((x,0),(y,0))$.
\end{proof}
\begin{corollary}
For each $d\in\mathbb{N}$ we have
$N(d)\ge d$
\end{corollary}

\begin{proof} It suffices to notice that
given a nonsingular bilinear map
$f\colon\mathbb{R}^d\times \mathbb{R}^d\to \mathbb{R}^k$, $f(x,\cdot)$ is injective, so
$\mathrm{rank}(f(x,\cdot))\ge d$.
\end{proof}
Similarly, we can obtain a somewhat more enlightening proof
that $N(4)=4$:
\begin{corollary}
\begin{align*}
	N(4)
&=
	4,\\
	N(8)
&=
	8.
\end{align*}
\end{corollary}
\begin{proof}
Thanks to the lower bound on $N(d)$, it suffices to prove
$N(4)\le 4,\ N(8)\le 8$.
$N(4)\le 4$ follows by considering the map  $f(x,y)=x\times_{\mathbb H} y$,
where $\times_{\mathbb H}$ denotes the quaternion multiplication.
$N(8)=8$ follows by considering the map $h(x,y)=
x\times_{\mathbb O} y$, i.e. the octonion multiplication.
\end{proof}
The lower bound can actually be strengthened:
\begin{theorem}\label{thm : lowerbound2}
We have
\begin{equation}
	N(d)
>
	 d
\end{equation}
unless $d\in \{1,2,4,8\}$, in which case equality holds.
We also have the general lower bound
\begin{equation*}
	N(d)
\ge
	2^{\lceil \log_2(d)\rceil}
\end{equation*}
for all $d$.
\end{theorem}
\begin{proof}
Assume $N(d)=d$ and let $f$ be a nonsingular bilinear map
that witnesses it.
Then  $(\mathbb{R}^d, +, f)$ is a real 
finite dimensional division algebra.
The only such algebras (see \cite[corollary 1]{Bott})
 have dimensions $\{1,2,4,8\}$ and the first
result follows.
The second result is proved as follows:
such a nonsingular bilinear map $f$ induces a map between real projective spaces
\begin{equation*}
	\tilde f\colon
	\mathbb{R}\mathbb P^{d-1}
\times
	\mathbb{R} \mathbb P^{d-1}
\to
	\mathbb{R}\mathbb P^{k-1},
\end{equation*}
so we have the following morphism between cohomologies:
\begin{equation*}
    \tilde f^*\colon 
    H(\mathbb R\mathbb P^{k-1}, \mathbb Z_2)\to 
	H(\mathbb R\mathbb P^{d-1}\times \mathbb{R} \mathbb P^{d-1}, \mathbb Z_2).
\end{equation*}
We recall that thanks to Künneth's theorem
(see \cite[Theorem 3.15]{hatcher2002algebraic}), we have
\begin{equation*}
H(\mathbb R\mathbb P^{d-1}\times \mathbb{R} \mathbb P^{d-1}, \mathbb Z_2)\simeq
H(\mathbb R\mathbb P^{d-1}, \mathbb Z_2)\otimes H(\mathbb R\mathbb P^{d-1}, \mathbb Z_2)
\end{equation*}
Using the fact that $H(\mathbb R\mathbb P^d, \mathbb Z_2)
\simeq \mathbb Z_2[X]/(X^{d+1})$
(\cite[Theorem 3.19]{hatcher2002algebraic}), we can equivalently write
\begin{equation*}
	\tilde f^*\colon 
	\mathbb Z_2[Z]/(Z^{k})
    \to 
    \mathbb Z_2[X]/(X^{d})
\otimes
	\mathbb Z_2[Y]/(Y^{d})\simeq \mathbb Z_2[X,Y]/(X^d,Y^d),
\end{equation*}
where $(X^d,Y^d)$ is the ideal
generated by $X^d, Y^d$.
By degree preservation
it follows that $\tilde f^*(Z)=aX+bY$ 
and by bilinearity $a=b=1$. We hence obtain that
\begin{equation*}
	(X+Y)^{k}
\in
	(X^d,Y^d),
\end{equation*}
where $(X^d, Y^d)$ is the ideal generated by $X^d, Y^d.$ By the binomial
theorem, it follows that this happens
if and only if
\begin{equation*}
    k\ge \min \left\{t\colon  \binom{t}{j}\equiv_20\text{ for all
    }t-d<j<d\right\}
\end{equation*}
The minimal such $k$ 
(which coincides with the right hand side) is denoted $d\circ d$ and
as one easily checks the binomial
characterisation implies $d\circ d\ge 2^{\lceil\log_2(d)\rceil}$.
\end{proof}
\begin{remark}
    A more detailed proof of $g(d)\ge 
    d\circ d$ is presented in
    \cite[Theorem 3.21]{hatcher2002algebraic} 
    \footnote{The theorem is only stated for $d=2^n$ but the method is easily
    generalised}.
    For  more details on this estimate,
    its extension to other fields
    etc, we refer the reader to
    \cite[Chapter 12]{shapiro}, and
    to \cite{Eliahou2005OldFunctions}
    for more information on the Hopf-Stiefel
    function $m\circ d$.
\end{remark}
\begin{remark}
The lower bound $g(d)\ge d\circ d$ is by no means sharp.
Indeed, if we define $\tilde g(d)$ as the minimal 
dimension such that a continuous nonsingular map
$f\colon\mathbb{R}^d\times \mathbb{R}^d\to \mathbb{R}^{\tilde g(d)}$
satisfying $f(x,-y)=-f(x,y)$ and $f(-x,y)=-f(x,y)$
exists\footnote{Such a map is called bi-skew} (without requiring linearity), the same bound
holds by the same proof.
\end{remark}
This result allows us to fully calculate $N(d)$ for $d\le 8$,
as shown in table \ref{tab:n8}.
\begin{table}[ht]
\centering
\caption{First $8$ values of $N(d)$}
\label{tab:n8}
\begin{tabular}{|c|c|}
\hline
\hline
$d$ & $N(d)$ \\
\hline
1 & 1  \\
2 & 2 \\
3 & 4  \\
4 & 4 \\
5 & 8 \\
6 & 8 \\
7 & 8 \\
8 & 8 \\
\hline
\end{tabular}
\end{table}

We turn to the upper bound on $N(d)$:
\begin{theorem}
For $d, k\in \mathbb{N}$ we have
\begin{equation*}
	N(d)
\le
	2d-1
\end{equation*}
and indeed more generally
\begin{equation*}
	N(kd)
\le
	(2k-1)N(d).
\end{equation*}
\end{theorem}
\begin{proof}
It suffices to prove that $g_{\mathbb K}(d)\le 2d-1$. 
To do so, we identify $\mathbb K^d\simeq \mathbb K_{d-1}[X]$,
the set of polynomials of degree $\le d-1$ over $\mathbb{K}$,
and then define $f$ as polynomial multiplication 
$
	\mathbb K_{d-1}[X]
\times 
	\mathbb K_{d-1}[X]
\to
	\mathbb K_{2d-2}[X]
\simeq
	\mathbb K^{2d-1}$. 
To prove the upper bound for $N(kd)$,
namely $N(kd)\le (2k-1)N(d)$,
it suffices to extend the previous polynomial multiplication trick
 to $(\mathbb{R}^d, f)$, where $f\colon\mathbb{R}^d\times \mathbb{R}^d\to \mathbb{R}^{g(d)}$ is a 
nonsingular bilinear map. 
More precisely, let us first identify 
$\mathbb{R}^{kd}\simeq \ell_2^k(\mathbb{R}^d)$,
the set of sequences
$(x_1,\dots, x_k)$ with each $x_i\in \mathbb{R}^d$
and define the map 
$
	h\colon
	\ell_2^k(\mathbb{R}^d)
\times
	\ell_2^k(\mathbb{R}^d)
\to
	\ell_2^{2k-1}(\mathbb{R}^{g(d)})$
as
\begin{equation*}
	h((x_1,\dots, x_k),(y_1,\dots, y_k))_l
=
	\sum_{i+j=l} f(x_i,y_j). 
\end{equation*}
To obtain the explicit inequality we claimed,
it suffices to notice that in our proof of equality of $N(d)$ and 
$g_{\mathbb{R}}(d)$, we can take $X=\ker(\tilde f),$
where $\tilde f$ is the map induced by $f$ on $M_d(\mathbb{R})
\simeq \mathbb{R}^d\otimes \mathbb{R}^d$. Making this explicit in our
case proves the inequality.
\end{proof}
As we will see later, the estimate $g_{\mathbb{R}}(d)\le 2d-1$ is not optimal.
On the other hand, the bound $g_{\mathbb K}(d)\le 2d-1$
cannot be improved in general:
\begin{proposition}\label{prop : gclosed}
Let $\mathbb K$ be an algebraically closed field. Then
\begin{equation*}
g_{\mathbb K}(d)=2d-1
\end{equation*} 
\end{proposition}
\begin{proof}[Proof (Folklore)]
To see this, it suffices to argue that given a subspace 
$V\subset M_d(\mathbb K)$ that avoids $\mathcal R_1$, then
$\mathbb P(V)\subset \mathbb P(M_d(\mathbb K))$ avoids
$\mathbb P(\mathcal R_1)$.
Since $\mathrm{dim}(\mathbb P(M_d(\mathbb K)))=d^2-1$
and $\mathbb P(\mathcal R_1)$ has dimension $2d-2,$
the result follows from the fact that the intersection
of two projective varieties $X,Y$ in $\mathbb P^d$
cannot be empty if $\mathrm{dim}(X)+\mathrm{dim}(Y)\ge d$
(see \cite[Theorem 7.2]{hartshorne1977ag}).
\end{proof}
The above upper bound $N(d)\le 2d-1$
can be turned into an explicit inequality:
\begin{theorem}\label{thm : KornHankel}
 Let $\mathcal H(d)$ be the set of Hankel matrices in $M_d(\mathbb{R})$,
 i.e. the set of matrices that are constant on skew-diagonals.
Then for all $u\in C^\infty_c(\mathbb{R}^d,\mathbb{R}^d)$
\begin{equation}\label{eq : Hankel-Korn}
	\|Q_{\mathcal H(d)}(\nabla u)\|_{L^2(\mathbb{R}^d)}
\le
	C(d)\|P_{\mathcal H(d)}(\nabla u)\|_{L^2(\mathbb{R}^d)}.
\end{equation}
\end{theorem}
For example, in $3$ dimensions the inequality becomes:
\begin{align*}
\|\nabla u\|_{L^2(\mathbb{R}^3)}^2
&\lesssim 
\|\partial_1 u_1\|_{L^2(\mathbb{R}^3)}^2+\|\partial_2 u_1+\partial_1 u_2\|_{L^2(\mathbb{R}^3)}^2+\|\partial_2 u_2\|_{L^2(\mathbb{R}^3)}^2\\ &+\|\partial_3 u_2+\partial_2 u_3\|_{L^2(\mathbb{R}^3)}^2+\|\partial_3 u_1+\partial_2 u_2+\partial_1 u_3\|_{L^2(\mathbb{R}^3)}^2.
\end{align*}
This improves on the standard $\mathrm{Sym}$ estimate
by using $\|\partial_3 u_1+\partial_2 u_2+\partial_1 u_3\|_{L^2(\mathbb{R}^3)}^2$ instead of 
$\|\partial_3 u_1+\partial_1 u_3\|_{L^2(\mathbb{R}^3)}^2
+\|\partial_2 u_2\|_{L^2(\mathbb{R}^3)}^2$.
\subsection{Improved estimates on \texorpdfstring{$g(d)$}{g(d)}}\label{sec : hardertopalg}
We now return to the problem of describing the behaviour
of $N(d)$, that is, that of $g(d)$. 
As we have mentioned before, we have the bounds $2^{\lceil \log_2(d)\rceil}\le 
g(d)\le 2d-1$ (for $d>1$, of course).
First, we mention the following
improved upper bound on $g(d),$
essentially due to 
\cite{lam1968construction}.
\begin{proposition}\label{prop : bindisj}
    Let $d\ge 2$. Then $g(d)\le 2d-2$.
\end{proposition}
\begin{proof}
We can apply Theorem 5.5 of \cite{lam1968construction} 
(with $h=d-1, k=2d-2$) to obtain
the existence
of a nonsingular bilinear map
\begin{equation*}
    f\colon\mathbb{R}^d\times \mathbb{R}^{d-1+\tau(2d-2,d-1)}
    \to \mathbb R^{2d-2},
\end{equation*}
where $\tau(2d-2,d-1)$ is defined
as follows:
\begin{equation*}
    \tau(2d-2,d-1)=\mathrm{Card}
    \left\{i\colon (d-1)_i=0\ \mathrm{and }\ (2d-2)_i\neq (d-1)_i\right\},
\end{equation*}
with $(a)_i$ denoting the $i$-th binary digit of $a$.
Since $\tau(2d-2,d-1)\neq 0$, the result
follows.
\end{proof}
\begin{remark}
    For $d$ even, the proof can be simplified substantially:
    identify $\mathbb{R}^d\simeq \mathbb{C}^{\frac d2}$
    and then use complex 
    polynomial multiplication.
\end{remark}
While the upper bound seems quite sharp
(being attained infinitely many times), the lower bound leaves much to be desired.
This is not surprising, as our proof only required very basic tools from algebraic
topology. To improve it, one needs some more sophisticated 
topological obstructions. Let us briefly sketch
the idea behind the lower bounds. The main ingredient one needs
is to make use of the connection of the problem with algebraic 
topology: namely, that the existence of a bilinear nonsingular map $f\colon\mathbb{R}^d\times \mathbb{R}^d\to \mathbb{R}^k$ implies the existence of an immersion of $\mathbb{R}\mathbb P^{d-1}$ in $\mathbb{R}^{k-1}$ (see \cite{Petrovic}). Hence $g(d)\ge \text{Imm}(\mathbb R\mathbb P^{d-1})+1,$ where $\mathrm{Imm}(M)$ is the minimal 
$k$ such that $M$ immerses in $\mathbb{R}^k$. 
We can now leverage the many results
concerning $\mathrm{Imm}(\mathbb{R}\mathbb P^d)$ \footnote{The following is by no means an exhaustive list, and we refer
the interested reader to \cite{davis_immtable} for a complete list of the bounds known} to obtain the remaining results:
\begin{theorem}\label{thm : algtop}The following bounds hold (see \cite[Theorem 1.4]{AsteyDavis1980} and \cite[Theorem 4.8]{Milnor} respectively): 
        \begin{align*}
        \mathrm{Imm}(\mathbb{R}\mathbb P^d)&=2d-\mathcal O(\log(d))\\
\mathrm{Imm}(\mathbb{R}\mathbb P^{2^n})&= 2^{n+1}-1
        \end{align*}
        \end{theorem}
It follows that
\begin{equation*}
    N(2^n+1)\ge \mathrm{Imm}(\mathbb{R}\mathbb P^{2^n})+1=2^{n+1}
\end{equation*}
and 
\begin{equation*}
    N(d)\sim 2d,
\end{equation*}
thus proving the remaining claims.
\begin{remark}
It is not known whether in general
$g(d)=\mathrm{Imm}(\mathbb{R}\mathbb P^{d-1})+1$.
Conditional on this equation being
true, the inequality $g_\mathbb{R}(d)\le 2d-2$
is attained if and only if $d=2^n+1$.\par
It seems clear that the problem of
determining additional values of $N(d)$
(and more ambitiously that of
determining a closed form for $N(d)$)
is fundamentally topological rather than analytical,
and one that will require new tools.
\end{remark}
For specific values, more can be said. Namely,  one can get
(see \cite{Adem}, \cite{DominguezLam2021} for proofs) the 
following larger table
\footnote{
For the numbers marked with $*$, 
the precise value is not known
and the table reports the best known bounds.
}:
\begin{table}[ht]
\centering
\caption{Values of $N(d)$}
\begin{tabular}{|cc|cc|cc|}
\hline
\hline
$d$ & $N(d)$ & $d$ & $N(d)$ & $d$ & $N(d)$ \\
\hline
1 & 1 & 12 & 17 & 23 & 39 \\
2 & 2 & 13 & 19 & 24 & 39 \\
3 & 4 & 14 & 23 & *25 & [40,47] \\
4 & 4 & 15 & 23 & *26 & [42,48] \\
5 & 8 & 16 & 23 & *27 & [46,49] \\
6 & 8 & 17 & 32 & *28 & [46,50] \\
7 & 8 & 18 & 32 & *29 & [47,51] \\
8 & 8 & 19 & 33 & *30 & [47,54] \\
9 & 16 & *20 & [33,35] & *31 & [47,54] \\
10 & 16 & 21 & 35 & 32 & 54 \\
11 & 17 & 22 & 39 & 33   & 64   \\
\hline
\end{tabular}
\end{table}
\subsection{The optimal constant \texorpdfstring{$C(X)$}{C(X)}}
While in the standard versions of Korn's 
inequality ($Y=\text{Sym}$ and $\text{Sym}_0$, for example),
one can usually calculate $C$ explicitly and prove
it independent of the dimension $d$, this is not the
case for Korn-Hankel, and indeed it is almost never the case
for inequalities induced by an admissible space $X$.
To see why, we need the following auxiliary lemma:
\begin{lemma}
Let $Z$ be a subspace of $M_d(\mathbb{R})$. Let $u,v$
be two independent random vectors that
 are uniformly distributed on 
 the unit sphere $\mathbb{S}^{d-1}$
and let $R=u\otimes v$. Then
\begin{equation*}
    \mathbb E(\|P_Z(R)\|^2)=\frac{\mathrm{dim}(Z)}{d^2}.
\end{equation*}
\end{lemma}
\begin{proof}
To see this, first notice that since
$
	\|P_Z(A)\|^2
=
	\sum_i \|P_{\mathrm{span}(x_i)}(A)\|^2
$
where the $x_i$ form an orthonormal basis of $Z$,
it suffices to prove the claim for a one-dimensional
space $Z$.
In that case, let $B$ be a unit vector in $Z$ and notice that
\begin{align*}
    \mathbb E(\|P_Z(R)\|^2)&=
     \mathbb E\left(\mathrm{Tr}(vu^tB)^2\right)
    \\&
    =
    \mathbb E(\mathrm{Tr}(u^t B v)^2)
    \\&
    =
    \mathbb E(\mathrm{Tr}(B^t uu^t B vv^t))
    \\&
    =
     \mathrm {Tr}(B^t\mathbb E(uu^t)B \mathbb E(vv^t))
    \\&
    =
   \frac1{d^2}  \mathrm{Tr}(B^tB)
    \\&
    =
    \frac1{d^2}.
\end{align*}
\end{proof}
This lemma allows us to investigate the dimension-dependence
of the constants in inequalities of Korn type as follows:
\begin{proposition}\label{prop : trivupperbound}
    Let $X$ be a subspace of $M_d(\mathbb{R})$
    and assume that, for all $u\in C^\infty_c(\mathbb{R}^d,\mathbb{R}^d)$
    we have
    \begin{equation*}
        \|Q_X(\nabla u)\|_{L^2(\mathbb{R}^d)}\le C(X) \|P_X(\nabla u)\|_{L^2(\mathbb{R}^d)}.
    \end{equation*}
    Then we must have
    \begin{equation*}
        C(X)^2\ge \frac{d^2}{\mathrm{dim}(X)}-1
    \end{equation*}
\end{proposition}
\begin{proof}
We have already seen that the Korn-type inequality holds
with constant $C(X)$ if and only if for all matrices $R$
of rank-one we have
\begin{equation*}
    C(X)^2\|P_X(R)\|^2\ge \|Q_X(R)\|^2.
\end{equation*}
If this holds for all matrices of rank-one, then it must follow that
\begin{equation*}
    \mathbb E\biggl(C(X)^2\|P_X(R)\|^2-\|Q_X(R)\|^2\biggr)\ge 0,
\end{equation*}
where $R$ is the same random variable as in the lemma.
Applying the lemma, the result follows.
\end{proof}
The previous proposition leaves open the question
of the sharpness of this bound. By applying concentration
inequalities, we prove that the bound is almost 
sharp for most randomly chosen $X$, 
provided we stay away from the critical range $\mathrm{dim}(X)\approx d$.
\begin{theorem}\label{thm : optimalconcentration}
    Let $k,d\in \mathbb{N}$ with $k\le d^2$. We define
    \begin{equation*}
        K(k,d):=\inf \biggl\{C(X)\colon X\subset M_d(\mathbb{R}), \mathrm{dim}(X)=k\biggr\}.
    \end{equation*}
    Then for any $\alpha>1$ if $k\gtrsim d^\alpha$ we have
    \begin{equation*}
        K(k,d)= \sqrt{\frac{d^2}{k}-1}+\mathcal O\left(\frac1{d^\beta}\right),
    \end{equation*}
    where $\beta\in(0,1)$ depends on $\alpha$.
\end{theorem}
\begin{proof}
    Let $\mathrm{Gr_k}(M_d(\mathbb{R}))$ denote the Grassmannian of order $k$ on $M_d(\mathbb R)$, i.e. the set
    of $k$-dimensional subspaces of $M_d(\mathbb R)$. As is well known
    (see \cite[Proposition 7.2]{Eaton2023}) we can equip this space
    with a Haar measure. 
    Let $X$ be a random variable which is uniformly distributed on $\mathrm{Gr}_k(M_d(\mathbb{R}))$ with respect
    to the aforementioned Haar measure, and let $A$
  be a matrix of norm $1$ in $M_d(\mathbb{R}).$
  It is not difficult to calculate that
  \begin{equation*}
      \mathbb E(\|P_X(A)\|^2)=\frac{k}{d^2}.
  \end{equation*}
  One way to see this is to notice first that, since 
  $P_{X}=\sum P_{x_i}$ where $\{x_i\}$ is an orthonormal
  basis of $X$, we can without loss of generality assume
  $X=\langle x\rangle$, and then rotational invariance 
  proves the result.
  By an elementary concentration inequality (see \cite{Dasgupta}) we have that 
  \begin{equation*}
      \mathbb P\left(\left|\|P_X(A)\|^2-\frac k{d^2}\right|>\varepsilon \frac k{d^2}\right)\lesssim \exp\left(\frac k2(\ln(1+\varepsilon)-\varepsilon)\right),
  \end{equation*}
      Moreover, since the map $(u,v)\to \|P_X(u\otimes v)\|^2$
      is Lipschitz as a map $\mathbb{S}^{d-1}\times_{\ell_2}\mathbb{S}^{d-1}\to \mathbb{R}$ (and its Lipschitz constant is
      easily bound by $2$ again), it follows that once we find a suitable $\frac{k}{d^2}\frac{\varepsilon}{100}$ net of $\mathbb{S}^{d}\times \mathbb{S}^d$, we
      obtain that for all matrices $R$ of rank one and norm 
      $1$ we have by the union bound
       \begin{equation*}
      \mathbb P\left(\left|\|P_X(R)\|^2-\frac k{d^2}\right|\ge \varepsilon\frac{k}{d^2}\right)\lesssim \mathcal N\left(\mathbb{S}^{d-1}\times_{\ell_2} \mathbb{S}^{d-1},\frac \varepsilon{1000}\frac{k}{d^2}\right)\exp\left(-\frac{k\varepsilon^2}{4}\right).
      \end{equation*}
      Since $\mathcal N\left( \mathbb{S}^d\times \mathbb{S}^d,\varepsilon \right)\lesssim (1+\frac 2\varepsilon)^{2d}$,
       it follows that
       \begin{equation*}
            \mathbb P\left(\left|\|P_X(R)\|^2-\frac k{d^2}\right|\ge \varepsilon\frac{k}{d^2}\right)\lesssim \exp(C_\varepsilon d\ln d)
            \exp\left(-\frac{k\varepsilon^2}{4}\right).
       \end{equation*}
       Hence if $k$ grows at least as fast as $d^\alpha$, the probability goes to $0$ as $d\to \infty$, so
      there exist subspaces $X\subset M_d(\mathbb{R})$ of dimension
      $k$ such that 
      \begin{align*}
          \inf_{R}\|P_X(R)\|^2&\ge \frac k{d^2} (1-\varepsilon)\\
          \sup_{R}\|P_X(R)\|^2&\le (1+\varepsilon)\frac k{d^2}
      \end{align*}
      The claim follows.
\end{proof}
\begin{remark}
     Similarly, we can prove that
      \begin{equation*}
        \left\lceil \frac{d^2}2\right\rceil\le   h(d)\le \left\lceil \frac{d^2}2\right\rceil+1,
      \end{equation*}
      where $h(d):=\min\{\dim(X)\colon  C(X)=1\}.$ For even
      $d$, it is not hard to construct
      an explicit subspace that attains $C(X)=1$ of dimension $\frac{d^2}2$, thus providing us
      with an optimal version of Korn's inequality with constant one:
      \begin{equation*}
          X=\left\{\begin{pmatrix}A&B\\
          -B&A\end{pmatrix}\colon A,B\in M_{\frac d2}(\mathbb{R})\right\}.
      \end{equation*}
      This space is not a subspace of $\mathrm{Sym}(d)$,
      and hence it is not a direct generalisation
      of the standard form of Korn's inequality, unless
      $d=2$. This, however, should not come as a surprise: it is not hard to prove that the only                                               
      subspace of $\mathrm{Sym}(d)$ (other than itself)
      for which $C(X)=1$ is $\mathrm{Sym}_0$.
\end{remark}
        It is not hard to see that our argument actually applies to a slightly
        larger region, i.e. $k\gtrsim_\varepsilon d\log(d)$. On the other
        hand, as we will see later, there are clear topological obstructions
        in the "small $k$" region $k\lesssim d$ (namely, $X^\perp \cap \mathcal R_1\neq \emptyset$
        becomes inevitable for $k$ slightly below $2d$), so the best we
        can hope for is for our argument to be extended to $k\gtrsim d$.
        It turns out that, up to a dependence on the constant on $\varepsilon$, this can be attained.\par
        To see this, notice that our argument introduced the $d\log(d)$ factor when we applied our naive
        $\epsilon$-net argument. If, instead of considering each individual rank-one matrix and then 
        finding a net, one considers all of them at the same time, this can be avoided, but the argument
        becomes more technical. More precisely, let us define $\gamma_2(T)$ for a set $T\subset S^{d-1}$
        to be
        \begin{equation*}
            \gamma_2(T):=\mathbb E\left(\sup_{x\in T} \langle g, x\rangle\right),
        \end{equation*}
        where $g$ is a gaussian vector, i.e. $g\sim \mathcal{N}(0,\mathrm{I}_d)$. This quantity, known
        as Gaussian width (see \cite[Chapter 9]{Vershynin2018} and references therein), is the exact right measure of size that
        is needed for our probabilistic argument to go through. Indeed we have that, for any 
        $k\gtrsim \frac{\gamma_2^2(T)}{\varepsilon^2}$ and letting $P_X$ denote the projection onto 
        a random element of $\mathrm{Gr}_k(\mathbb{R}^d)$ we have (see \cite[Section 4]{Klartag2005})
        \begin{equation*}
            \mathbb{P}\left(\text{for all } x\in T,\ (1-\varepsilon)
            \le \|P_X(x)\|\frac{d}{\sqrt k}\le (1+\varepsilon)\right)>\frac 12.
        \end{equation*}
        To apply this to $T:=\mathcal R_1\cap S_{M_d(\mathbb{R})}$, we wish to calculate $\gamma_2(T)$.
        This is not hard to do:
        \begin{equation*}
            \gamma_2(T)=\mathbb E\left(\sup_{\|x\|=\|y\|=1}\langle g, x\otimes y\rangle\right)
            =\mathbb E(\|g\|_{\mathrm{Op}})\approx \sqrt d.
        \end{equation*}
        It follows that the argument in the proof of Theorem \ref{thm : optimalconcentration}
        can be extended to the region
        \begin{equation*}
            k\gtrsim \frac{\gamma_2^2(T)}{\varepsilon^2}=\frac{d}{\varepsilon^2}.
        \end{equation*}
        In other words, on the region $d=o(k)$ we have, as $d$ goes to infinity,
        \begin{equation*}
            K(k,d)=\sqrt{\frac{d^2}{k}-1}+o(1)
        \end{equation*}
\begin{remark}
    In case one is interested, more 
    broadly, in linear operators 
    $T(\nabla u)$
    such that $\|T(\nabla u)\|_{L^2(\mathbb{R}^d)}\gtrsim 
    \|\nabla u\|_{L^2(\mathbb{R}^d)},$ while
    maintaining\footnote{Such a normalisation is 
    necessary, for otherwise 
    it would suffice to consider $P_X/k$ for $k$ large enough.} $\|T\|_{F}\gtrsim 
    d$ then it is 
    possible to obtain
    such an operator with a much smaller
    range than in the projection case.
    In particular, a nonsingular bilinear
    map is said to be normed if $\|f(x,y)\|=\|x\|\|y\|$. In that case,
    (let $F$ denote the extension of $f$ to $M_d(\mathbb{R})$) the same rank-one convexity argument
    as in section \ref{sec : calcvar} proves that
    \begin{equation*}
        \|F(\nabla u)\|_{L^2(\mathbb{R}^d)}\ge \|\nabla u\|_{L^2(\mathbb{R}^d)},
    \end{equation*}
    and it is not hard to prove that
    $\|F\|=1$.
    Not much is known about the minimal dimension $\tilde g(d)$ such that a normed
    bilinear nonsingular map $f\colon\mathbb{R}^d\times \mathbb{R}^d\to \mathbb{R}^{\tilde g(d)}$ exists: to the author's knowledge only
    the following values are known: for $d\le 10,$
    $g(d)= \tilde g(d)$ and $\tilde g(16)\in [29,32]$. Asymptotically, one 
    can prove that $\tilde g(d)=\mathcal O\left(\frac{d^2}{\log(d)}\right)$, but
    it is not known if this is optimal.
    For other rings
    and fields, more is known: for $\mathbb K=\mathbb{C}$ (and for all fields $\mathbb K$ with positive characteristic) we
    have $\tilde g_{\mathbb{C}}(d)=\mathcal O(d^{1.62})$; lower-bound-wise
    we have $\tilde g_{\mathbb Z}(d)=\Omega(d^{\frac 65}).$
\end{remark}
It follows that the only way the constant can be uniform in dimension
is if $\mathrm{dim}(X)\approx d^2$, and that $C(\mathcal H(d))\gtrsim \sqrt d$. Somewhat surprisingly (compared to its parent
case $X=\mathrm{Sym}$, which is quite close to being optimal), $C(\mathcal H(d))$ grows much faster than polynomially. 
Indeed we can, with a more refined analysis,
pinpoint $C(\mathcal H(d))$ quite precisely:
\begin{proposition}
\begin{equation*}
    \log(C(\mathcal H(d)))=\frac{2G}\pi d+\mathcal O(\log(d)).
\end{equation*}
where $G$ is Catalan's constant.
\end{proposition}
\begin{proof}
To see this, notice that, given two vectors $u$ and $v$ in $\mathbb{R}^d$, 
$ P_{\mathcal H(d)}(u\otimes v)=u\star v,$ where $u\star v$ is the polynomial coefficient convolution
operation induced by $\mathbb{R}^d\simeq \mathbb{R}_{\le d-1}[X]$ and $u*v$ denotes the standard convolution operation. Hence the problem of determining $C(\mathcal H(d))$
is equivalent to that of determining
\begin{equation*}
    C(\mathcal H(d))^{-1}\approx \inf_{u,v\in \mathbb{R}^d\setminus\{0\}} \frac{\|u\star v\|}{\|u\|\|v\|}.
\end{equation*}
Notice moreover that $1/\sqrt{d}\|u*v\|\le \|u\star v\|\le  
\|u*v\|$.
Using this and selecting the coefficients properly, it is not hard to obtain an exponential bound\footnote{One can just take $u=\sum_i \binom{d}{i}e_i$ and $v=\sum_i\binom{d}{i}(-1)^i e_i$; then $u*v=\sum_k (-1)^k\binom{d}{k}e_{2k}$
and it follows that $C(\mathcal H(d))^{-1}\le \frac{1}{\sqrt{\binom{2d}{d}}}$.};
this can be improved by noticing that, by Parseval, we can rephrase the problem as one about the $L^2(\mathbb{T})$
norm of polynomials with real coefficients:
\begin{equation*}
   \inf_{u,v\in \mathbb{R}^d\setminus\{0\}}\frac{\|u*v\|}{\|u\|\|v\|}=\underset{\substack{\deg(p),\deg(q)\le d-1}}{\inf}\frac{\|p(z)q(z)\|_{L^2(\mathbb{T})}}{\|p(z)\|_{L^2(\mathbb{T})}\|q(z)\|_{L^2(\mathbb{T})}}.
\end{equation*}
Since, for a polynomial in 
$\mathbb{R}_{d-1}[X]$, we have
that the $L^2(\mathbb T)$ and $L^\infty(\mathbb T)$ norms are equivalent
up to a polynomial factor, i.e.
\begin{equation*}
    \frac1{\sqrt d}\|p\|_{L^\infty(\mathbb T)}
    \le \|p\|_{L^2(\mathbb T)}\le
    \|p\|_{L^\infty(\mathbb{T})},
\end{equation*}
we have
\begin{equation*}
    \log\left(C(\mathcal H(d))\right)=\underset{\substack{\deg(p),\deg(q)\le d-1}}{\sup} \log\left(\frac{\|p\|_{L^\infty(\mathbb{T})}\|q\|_{L^\infty(\mathbb{T})}}{\|pq\|_{L^\infty(\mathbb{T})}}\right)+\mathcal O(\log(d)).
\end{equation*}
The value of the right hand side is already known, see \cite[Theorem 2, Remark 2]{Boyd}
and \cite[corollary 2.8, Theorem 2.9, Remark 2.10]{RumpSekigawa2010}:
\begin{equation*}
   \underset{\substack{\deg(p),\deg(q)\le d-1}}{\sup} \log\left(\frac{\|p\|_{L^\infty(\mathbb{T})}\|q\|_{L^\infty(\mathbb{T})}}{\|pq\|_{L^\infty(\mathbb{T})}}\right) =d\log(\delta)
    +\mathcal O(\log(d)),
\end{equation*}
where $\delta=e^{2G/\pi}$, proving the result.
\end{proof}

This leaves open the question whether, in the regime
$\dim(X)\approx 2 d,$ an optimality result such as Theorem
\ref{thm : optimalconcentration} can be attained. While the optimal regime
remains outside of our reach, if one is content
with only near-optimality, then the
result can be substantially improved:
\begin{proposition}
    For $k\ge (4+\delta)d,$ with $\delta>0$, there exists
    a subspace $X$ of $M_d(\mathbb R)$ of dimension $k$
    such that $C(X)\approx_\delta \sqrt d$.
\end{proposition}
\begin{proof}
    Let $G$ be a Gaussian matrix
    that maps onto $X$. Then 
    by Gordon's min-max inequality (see \cite[Theorem A, corollary 1.2]{Gordon})
    we have
    \begin{equation*}
        \mathbb E\left(\min_{A\in \mathcal R_1\cap S}\| GA\|\right)\gtrsim 
        \sqrt k-\gamma_2(\mathcal R_1\cap S)
        \approx \sqrt k-2\sqrt d.
    \end{equation*}
    It is well known that with high probability
    \begin{equation*}
        d-\sqrt k \lesssim \sigma_1(G)\lesssim \sqrt{k}+d,
    \end{equation*}
    so in our regime $\sigma_1(G)\approx d$.
    It follows that
    \begin{equation*}
        \mathbb E\left(\min_{A\in\mathcal R_1\cap S}\|P_X(A)\|\right)\gtrsim\frac{1}{d} \mathbb E\left(\min_{A\in\mathcal R_1\cap S}\|G(A)\|\right)
        \gtrsim\frac{\sqrt{k}-2\sqrt d}{d}.
    \end{equation*}
    This, together with proposition \ref{prop : trivupperbound} implies that there exists a subspace
    $X$ such that
    \begin{equation*}
        C(X)\approx \sqrt d.
    \end{equation*}
    Indeed, by applying standard concentration arguments
    it follows that this happens with high probability.
\end{proof}
\begin{remark}
    Gordon's theorem cannot be pushed any 
    further than $k>(4+\delta) d,$ as the estimate
    above becomes negative and hence useless. This
    is no coincidence: the statistical dimension of the convex hull
    generated by rank one matrices of norm $1$, i.e. the unit ball with respect to the
    nuclear norm, has statistical dimension $4d$, so from the viewpoint of
    these Gaussian arguments, it is not possible to guarantee that $X^\perp$
    does not intersect $\mathcal R_1$.
\end{remark}
\subsection{
Extending to \texorpdfstring{$p\neq 2$}{p!=2}:
\texorpdfstring{$N_p(d)$}{Np(d)}
}
We have so far been interested only in the dimensionality
problem in $W^{1,2}$, but the results above extend
to $W^{1,p}$ for $p\in(1,\infty).$ More precisely
we can define $N_p(d)$ in the same way as $N(d),$
just with $p$ instead of $2$ in \eqref{eq : Ndef}:
\begin{definition}\label{def : Np}
	Let $p\in (1,\infty)$. We define $N_p(d)$ to be
the smallest integer $k$ for which
there exist $\ell_1,\dots, \ell_k$ functionals on $M_d(\mathbb{R})$
such that
\begin{equation}
	\label{eq : Npdef}
	\|\nabla u\|_{L^p(\mathbb{R}^d)}
\lesssim
	\sum_{i=1}^k \|\ell_i(\nabla u)\|_{L^p(\mathbb{R}^d)}
\end{equation}
holds for all $u\in C^\infty_c(\mathbb R^d,\mathbb{R}^d)$.
\end{definition}
In the same way as for $N_2(d)$ we have:
\begin{lemma}
Let $M_p(d)$ be defined as
\begin{equation*}
	M_p(d)
:=
	\inf
	\left\{
		\mathrm{dim}(X)\colon  
		\forall u\in C^\infty_c(\mathbb{R}^d,\mathbb{R}^d)
		\  \|\nabla u\|_{L^p(\mathbb R^d)}
	\lesssim
		\| P_X(\nabla u)\|_{L^p(\mathbb R^d)}
	\right\}.
\end{equation*}
Then $M_p(d)=N_p(d)$.
\end{lemma}
Mutatis mutandis, we can follow the proof of Theorem \ref{thm : solchipot} to obtain
$N_p(d)\ge g(d),$ but the opposite direction of the proof
does not go through any more (since being rank-one convex
is only necessary for quasiconvexity). 
However, since we do not care about the precise constant,
a different argument can be employed to prove
the upper bound as well.
\begin{theorem}
For $p\in (1,\infty)$ we have
   \begin{equation*}
        N_p(d)= N_2(d).
    \end{equation*}
\end{theorem}
\begin{proof}
We already have $N_p(d)\ge g(d),$ so it suffices to prove
$N_2(d)\ge N_p(d)$. To prove it,
let $X$ be an admissible subspace for $N_2(d)$.
Then the operator $T(u):= P_X(\nabla u)$
is an elliptic operator for $p=2$, so it is
 elliptic in the injective sense
(see \cite[Theorems 5.1, 5.3]{Schaftingen}),
hence we have the bound
\begin{equation*}
	\|\nabla u\|_{L^p(\mathbb R^d)}
\lesssim_{T,d,p}
	\|T u\|_{L^p(\mathbb R^d)},
\end{equation*}
which proves the claim.
\end{proof}
\begin{remark}
	The same questions that we posed for $p=2$
	(the best decay that can be obtained with $\text{dim(X)}
	=N_p(d)$ and the best $\mathrm{dim}(X)$ that can be attained
	while maintaining a constant $\ge 1$) can be asked
	for $p\neq 2$. In this case, we remark that the author
	\cite{Cassese} has obtained dimension-free bounds
	for various standard forms of Korn's inequality.
\end{remark}
While for $p=2$, the issue of determining the precise value of
the constant $C(p,X)$ such that
\begin{equation*}
	\|Q_X(\nabla u)\|_{L^p(\mathbb R^d)}
\le
	C(p,X) \|P_X(\nabla u)\|_{L^p(\mathbb R^d)}
\end{equation*}
holds can be explicitly answered (at least a priori)
 using
\begin{equation*}
     C(X,2)^{-1}=\inf_{A\in\mathcal R_1} \frac{\|P_X(A)\|}{
\|Q_X(A)\|},
\end{equation*}
the situation is much more complicated for
$p\neq 2$. Indeed, the question of determining the $L^p$
constant is still open even for the standard version of Korn's
inequality (see \cite{Cassese} for more on this). However,
we can employ our method to obtain bounds on a similar constant,
namely $C^{\mathrm{rc}}$, which is defined as
\begin{equation*}
	C^{\mathrm{rc}}(X,p)
:=
	\inf
	\left\{
		c\colon 
		 f_{p,c,X}^{rc}(0)
	=
		0
	\right\}.
\end{equation*}
Indeed, Theorem \ref{thm : mainthm} allows us to prove
\begin{proposition}
\begin{equation*}
	C^{\mathrm{rc}}(X,p)
\le
	C(X,2)(p^*-1).
\end{equation*}
\end{proposition} 
Focusing in particular on the Hankel-Korn inequality,
we can take advantage
of the  sharp part of Theorem \ref{thm : sharp}.
Indeed, we can embed the scalar martingales by defining 
$M_n=\frac{f_n}2 (e_1\otimes e_2+e_2\otimes e_1)+\frac{g_n}
2(e_1\otimes e_2-e_2\otimes e_1)$
and the theorem then allows us to show 
\begin{equation*}
	C^{\mathrm{rc}}(\mathcal H(d), p)
\ge
	(p^*-1).
\end{equation*}
It is then natural to conjecture
\begin{conjecture}
\begin{equation}
	C(\mathcal H(d),p)
=
	C^{\mathrm{rc}}(\mathcal H(d),p)
=
	C(\mathcal H(d),2)(p^*-1).
\end{equation}
\end{conjecture}
Let us point out that, for $d=2$, we have 
$C^{\mathrm{rc}}(\mathcal H(2),p)=
C(\mathcal H(2), 2)(p^*-1)$, so the conjecture becomes 
one of quasiconvexity at $0$ of
$f_{\mathcal H(2), C^{\mathrm{rc}},p}^{rc}$
and that in this case the conjecture is 
similar to (and indeed, implied by)
the Iwaniec-Martin conjecture
(see \cite{Iwaniec1982OnExponents}).
Indeed, Theorem \ref{thm : sharp} can be used to obtain
a new proof of the lower bound in the 
Iwaniec-Martin conjecture.
This will be further
explored in future work.
\begin{remark}
The study of the quasiconvexity of these rank-one convex envelopes
is an active area of research, particularly for $X=\mathcal C(2)$ (i.e. the set of 
conformal $2\times 2$ matrices)
thanks to its relation with the Iwaniec-Martin conjecture (see \cite{BaernsteinII1997}, \cite{Iwaniec2002Nonlinear}, \cite{Astala2012BurkholderMappings}
for more on this connection). 
In this case, it turns out that the envelope can be explicitly 
described in terms of the Burkholder function, whose conjectured quasiconvexity
has been studied intensively, see \cite{Guerra2022}, \cite{Astala2023TheTheory}, \cite{Guerra2019}, \cite{AstalaIwaniecPrauseSaksman2015} and \cite{KariAstala2024}.    
\end{remark} 
\section{Korn's second inequality:
\texorpdfstring{$N'(d)$}{N'(d)}}\label{sec : second}
We now deal with $N'(d)$, i.e. the minimal dimension of
$X\subset M_d(\mathbb{R})$ such that for all
$u\in H^1(\Omega)$ we have
\begin{equation*}
	\|u\|_{H^1(\Omega)}
\lesssim
	\|u\|_{L^2(\Omega)}
+
	\|P_X(\nabla u)\|_{L^2(\Omega)}.
\end{equation*}
First, let us note that the correct definition should read 
$N'(d,\Omega)$, as it can a priori depend on the domain. 
We will use the term $N'(d)$ to refer to $N'(d,\Omega)$
for any open bounded set $\Omega$ with Lipschitz boundary.
\footnote{
	One can extend this to, say, John domains, 
	but for simplicity we stick with this simpler case.
	Let us note that the geometry of $\Omega$ does play a role:
	namely, if $\partial \Omega$ is too irregular, 
	Korn's second inequality fails.
}
In \cite{Chipot2021OnType} the following
was proved:
\begin{theorem}
    \begin{enumerate}
        \item $N'(2)=3$
        \item $N'(d)\ge N(d)$
        \item $N'(d)\ge d$ and the inequality is strict if $d$ is odd and $d>1$.
    \end{enumerate}
\end{theorem}
We prove a closed form for $N'(d).$ Namely
\begin{equation*}
    N'(d)
=
	g_{\mathbb C}(d)
=
	2d-1.
\end{equation*}
This implies the previous theorem in a strengthened form,
since both inequalities stated above
have to be strict (provided $d>1$).
\begin{proof}[Proof of $N'(d)\ge g_{\mathbb C}(d)$]
Let $X$ be admissible for $N'(d)$, i.e. let us assume that
\begin{equation*}
	\|u\|_{H^1(\Omega)}
\lesssim
	\|u\|_{L^2(\Omega)}
+
	\|P_X(\nabla u)\|_{L^2(\Omega)}
\end{equation*}
holds. Assume, for the sake of contradiction,
that $\mathrm{dim}(X)<g_\mathbb{C}(d)$.
It follows that ${X^{\mathbb{C}}}^\perp$
(i.e. the orthogonal complement of the complexification of $X$)
 contains matrices of rank $1$,
hence there exist $a,b\in \mathbb{C}^d$ such that
\begin{equation*}
	P_{X^\mathbb{C}}(a\otimes b)
=
	0.
\end{equation*}
Let us define the set $S:=\{u\in W^{1,2}(\Omega)\colon P_X(\nabla u)=0\}$.
We claim $S$ is infinite-dimensional. 
To see this, take $g\colon \mathbb{C}\to \mathbb{C}$ entire, and define
\begin{equation*}
    u_g(x)=\Re(g(b\cdot x)a).
\end{equation*}
It follows that
\begin{equation*}
    \nabla u_g=\Re(g'(b\cdot x)a\otimes b),
\end{equation*}
hence $u_g\in S$, so $S$ is infinite-dimensional. 
Since $X$ is admissible, it follows that for all $u\in S$ we have
\begin{equation*}
    \|u\|_{W^{1,2}(\Omega)}\lesssim \|u\|_{L^2(\Omega)},
\end{equation*}
so the two norms $\|\cdot\|_{H^1}$ and $\|\cdot \|_{L^2}$
are equivalent on $S$.
Since the embedding $H^1(\Omega)\to L^2(\Omega)$ is compact,
 it follows that the unit ball of $S$ is totally bounded,
which implies $S$ is finite-dimensional,
a contradiction.
\end{proof}
\begin{proof}[Proof of $N'(d)\le g_{\mathbb{C}}(d)$]
We prove the result for domains
that are star-shaped with respect to a ball.
Standard partition of unity and 
reflection arguments extend the result to Lipschitz domains.
Let $X$ be a subspace of $M_d(\mathbb{R})$ such that
$(X^{\perp})^{\mathbb C}$ avoids all rank one matrices. Let $T_X$ denote the 
differential operator $T_X(u)=P_X(\nabla u)$. 
Thanks to the rank-one condition, 
we have that the family $T_X(\cdot)_{i,j}$ of scalar differential operators satisfies condition $(C)$
of \cite{Kalamajska}, hence thanks to \cite[Theorem 4]{Kalamajska} we obtain that, for each $u\in C^\infty(\Omega)$ we have
\begin{equation*}
    u(x)=\mathcal P(u)(x)+\int_{\Omega}K(x,y)T_X(u)(y)
    \mathrm{d} y,
\end{equation*}
where $K$ is a vector-valued Calderón-Zygmund kernel and 
$\mathcal P(u)(x)$ is defined as
\begin{equation*}
    \mathcal P(u)(x)=
    \int_{\Omega}\sum_{|\alpha|\le l}
    \partial^\alpha_y\left(\frac{(x-y)^\alpha}{\alpha!}\omega(y)\right)
    u(y)\mathrm{d} y,
\end{equation*}
where $\omega$ is a mollifier
supported in $B(0,1)$ and 
$l\in\mathbb{N}$ depends on $\Omega, T$
and $\omega$ but not on $u$.
It follows that
\begin{equation*}
    \|u\|_{H^1(\Omega)}
\lesssim_{d,\Omega, X}
    \|T_X(u)\|_{L^2(\Omega)}+\|u\|_{L^2(\Omega)}
\end{equation*}
and the result follows by density.
\end{proof}
One can define $N'_p(d, \Omega)$ in the same manner
as $N_p$ was defined from $N$. The same argument
as above then proves that 
$N'_p(d,\Omega)=N_2'(d,\Omega)=2d-1$.
\begin{remark}
    The above results can also be interpreted (and proved) in terms
    of $\mathbb C$-ellipticity, a strengthening of
    the concept of ellipticity. Namely, a differential operator
    of order $k$ of the form
    \begin{equation*}
        T(u)=\sum_{|\alpha|=k}
        T_\alpha \partial_\alpha u
    \end{equation*}
    
    having symbol $T(\xi)=\sum
    T_\alpha \xi^\alpha$ is
    said to be $\mathbb{C}$-elliptic if 
    $\mathrm{ker}(T(\xi))=0$
    for all $\xi\in \mathbb{C}^n\setminus \{0\}$,
    and if it is then the inequality 
    $\|u\|_{W^{k,p}(\Omega)}\lesssim \|u\|_{L^p(\Omega)}+\|T(u)\|_{L^p
    (\Omega)}$ holds
    for suitably well behaved domains. It is not 
    hard to see that, in our context, $T_X$ being $\mathbb{C}$-elliptic
    is equivalent to ${X^\perp}^\mathbb{C}$ not containing 
    any rank-one matrices. The concept of $\mathbb C$-ellipticity
    has been independently discovered many times:
    Smith \cite{smith1961inequalities}
    first obtained a result of this kind, followed by 
    De Figueiredo \cite{DeFigueiredo}, Boman \cite{Boman},
    Kałamajska \cite{Kalamajska} and more recently
    by Breit, Diening, and Gmeineder \cite{Gmeineder}.
    Boman's proof, in particular, implies a strengthening of our 
    theorem to Boman domains, which as is known (\cite{BomanJohn}) coincide
    with John domains. Since Korn's 
    second inequality cannot in general be extended to domains that are not John (\cite{JohnKorn}), this is in a sense the
    largest class possible.
\end{remark}
In particular, we have
\begin{theorem}[Second Korn-Hankel inequality]
    Let $\Omega$ be a Lipschitz domain in $\mathbb{R}^d$. 
   Then for all $p\in(1,\infty)$
   and $u\in W^{1,p}(\Omega, \mathbb{R}^d)$
   we have
    \begin{equation*}
        \|u\|_{W^{1,p}(\Omega)}
\le
	C'(\Omega, p)
	\left(
		\|u\|_{L^p(\Omega)}
	+
		\|P_{\mathcal H(d)}(\nabla u)\|_{L^p(\Omega)}
	\right)
    \end{equation*}
\end{theorem}
\begin{remark}
	The issue of determining the constants
	associated with the inequality induced by 
	an admissible subspace $X$ is much more subtle
	in this case than it was for $N(d)$: the geometry
	of the domain plays a very important role, as shown in
	 \cite{Lewicka2016OnConditions}. More precisely,
	it is known that for the standard form of Korn's 
    second inequality, 
	the constant is not uniformly bounded with respect to
	(Lipschitz) domains. This is not surprising,
    as one expects for physical reasons that
    the constant should blow up for very thin strips.\par
    It is not clear what geometric quantities
    one should expect $C'(\Omega, X,p)$ to depend on.
\end{remark}
\section{The rectangular case: \texorpdfstring{$N(m,d)$}{N(m,d)}}\label{sec : rect}
The Korn-type inequalities
we have considered so far
have one restriction: namely, they only
apply to vector fields $u\colon \mathbb{R}^d\to \mathbb{R}^d$.
The goal of this section is to extend
our results to vector fields $u\colon \mathbb{R}^m\to \mathbb{R}^d$, or in other words to the spaces
$H^1_{(0)}(\Omega\subset \mathbb{R}^m; \mathbb{R}^d)$.
\begin{definition}[$N(m,d)$]
    We define $N(m,d)$ to be
the smallest integer $k$ for which
there exist linear $\ell_1,\dots, \ell_k$ functionals on $M_{m,d}(\mathbb{R})$
such that
\begin{equation}
	\label{eq : Nmddef}
	\|\nabla u\|_{L^2(\mathbb{R}^m)}^2
\lesssim
	\sum_{i=1}^k \|\ell_i(\nabla u)\|_{L^2(\mathbb{R}^m)}^2
\end{equation}
holds for all $u\in C^\infty_c(\mathbb R^m,\mathbb{R}^d)$.
\end{definition}

One similarly defines $N'(m,d)$, with
the usual caveats on $\Omega$ being a regular-enough domain,
and $n_{\mathbb K}(m,d,k), g_{\mathbb K}(m,d)$. 
The same method we used in the previous sections allows us to 
prove the following theorem
\begin{theorem}
    Let $m,d\in\mathbb N$. Then
    \begin{equation*}
        N(m,d)=md-n_{\mathbb R}(m,d,1)=g_{\mathbb R}(m,d)
    \end{equation*}
    and
    \begin{equation*}
        N'(m,d)=md-n_{\mathbb C}(m,d,1)=
        g_{\mathbb C}(m,d).
    \end{equation*}\end{theorem}
\begin{remark}
This also provides an easy proof
that both $N(m,d)$ and $N'(m,d)$
are symmetric with respect to $m$ and $d$ (since $g_{\mathbb K}$ is)
which, while not surprising
(especially in light of the easily
provable fact that $N(d,1)=N(1,d)=d$),
    has proved rather elusive to 
    prove directly.
\end{remark}
The same argument as in 
Proposition \ref{prop : gclosed} 
proves
\begin{equation*}
    N'(m,d)=m+d-1.
\end{equation*}
Moreover, we have the following
bounds:
\begin{theorem}
Let $m,d\in\mathbb{N}$. Then
\begin{equation*}
    \max\{d,m\}\le m\circ d\le N(m,d)\le
    N'(m,d).
\end{equation*}    
\end{theorem}
\begin{proof}[Sketch]
    It is clear by the polynomial multiplication bound
    that 
    $g_{\mathbb{R}}\le g_{\mathbb{C}},$ so the rightmost
    upper bound follows. To prove
    that $N(m,d)\ge m\circ d,$
    we adapt the proof of Theorem \ref{thm : lowerbound2}: a nonsingular bilinear map $f\colon\mathbb{R}^m\times\mathbb{R}^d\to \mathbb{R}^{g(m,d)}$
    implies the existence of a
    morphism $\tilde f$
    \begin{equation*}
        \tilde f\colon
        \mathbb Z_2[Z]/(Z^{g(m,d)})\to \mathbb Z_2[X]/(X^m)
        \otimes \mathbb Z_2[Y]/(Y^d)
    \end{equation*}
    satisfying $\tilde f(Z)=X+Y,$
    hence
    \begin{equation*}
        (X+Y)^{g(m,d)}\in (X^m,Y^d),
    \end{equation*}
    i.e. $g(m,d)\ge m\circ d$. 
    The left-most lower bound follows
    by standard algebraic considerations.
\end{proof}
Following the
proof of Proposition \ref{prop : bindisj}, we can also explicitly
characterise the values of $m,d$ for which the equality $N(m,d)=N'(m,d)$ holds:
\begin{proposition}
    Let $m,d\in \mathbb{N}$. Then $N(m,d)=N'(m,d)=m+d-1$ if and only if
    $m-1,d-1$ are dyadically disjoint,
    i.e. $\{i\colon (m-1)_i=1\}\cap \{i\colon  (d-1)_i=1\}=\emptyset.$
\end{proposition}
\begin{proof}
    As in Proposition \ref{prop : bindisj}, we employ Theorem 5.5
    of \cite{lam1968construction}
    to obtain the existence
    of a nonsingular bilinear map
    \begin{equation*}
        f\colon \mathbb{R}^m\times \mathbb{R}^{d-1+\tau(m+d-2,m-1)}\to \mathbb{R}^{m+d-2},
    \end{equation*}
    where
    \begin{equation*}
        \tau(m+d-2,m-1)=\mathrm{Card}
        \left\{i\colon  (d-1)_i=0\ \mathrm{and}\ (m+d-2)_i\neq (m-1)_i\right\}.
    \end{equation*}
    It follows that, if $m-1,d-1$
    are not dyadically disjoint then $\tau>0$. On the other hand, if
    they are dyadically disjoint then
    $m\circ d=m+d-1$, and since $m\circ d\le g(m,d)\le m+d-1$, the result follows.
\end{proof}

\par
As in section \ref{sec : hardertopalg},
we can use more
involved topological arguments to prove stronger lower bounds, e.g.\footnote{We make no claim 
of sharpness for this lower bound, which we
present only as an example of the bounds that can be attained by using more topological tools.}
\begin{equation*}
    g_{\mathbb R}(m,d)\ge \mathrm{min}\left(\mathrm{Imm}(\mathbb{R}\mathbb P^{m-1}), \mathrm{Imm}(\mathbb{R}\mathbb P^{d-1})\right)+1.
\end{equation*}
For more on bounds on $g_{\mathbb R}(m,d)$ (and many explicit
values of it), we refer the reader
to \cite{DominguezLam2021} and references
therein. \par
Finally, we mention that the 
 Korn-Hankel inequalities hold in this context 
 as well, with virtually no modifications
 necessary to either the argument or the formulation of the inequalities. Moreover,
 by using the same trick of reframing
 the matter as one concerning the norm
 of polynomials, one can prove
 \begin{equation*}
      C(\mathcal H(m,d),2)
      =\frac1{\mu_{m+d,d}^{\mathbb R}}\mathcal O(dm),
 \end{equation*}
 where $\mu_{m+d,d}^{\mathbb R}$ is defined
 as
 \begin{equation*}
     \mu_{m+d,d}^{\mathbb R}
     =\inf\left\{\frac{\|pq\|_{L^\infty(\mathbb{T})}}{\|p\|_{L^\infty(\mathbb{T})}\|q\|_{L^\infty(\mathbb{T})}}\colon  
     p\in \mathbb{R}_{m-1}[X], q\in \mathbb{R}_{d-1}[X]\right\}.
 \end{equation*}
 Bounds similar (albeit more complicated
 to state) to the ones obtained
 in the case $d=m$ can then be proved by 
 taking advantage of the known 
 properties of $\mu$; we refer the 
 interested reader to \cite{burger}
 for details.
 \section{A quantitative proof of Ornstein's non-inequality}\label{sec : Ornstein}
 In \cite{Ornstein}, Ornstein proved the following result:
 \begin{theorem}
 Let $L$ be a homogeneous differential operator of order $1$ with 
 constant coefficients, i.e. $L(u)=T(\nabla u)$ for all $u\in C^\infty_c(\mathbb{R}^d, \mathbb{R}^m)$, with $T\in \mathrm{Lin}(M_{m,d}(\mathbb{R}), \mathbb{R}^k)$. Then there exists a constant $C$ such that
 for all $u\in C^{\infty}_c(\mathbb{R}^d, \mathbb{R}^m)$
 \begin{equation*}
     \|\nabla u\|_{L^{1}(\mathbb{R}^d)}\le C \|L(u)\|_{L^1(\mathbb{R}^d)}
 \end{equation*}
 if and only if $T$ is injective.
 \end{theorem}
 To be precise, Ornstein proved this result for higher orders as well, but here we focus
 on operators of order $1$. This theorem has a clear connection
 with calculus of variations, for it can be equivalently stated in terms
 of quasiconvexity: namely, let $f_{T,C}(A):=C\|T(A)\|-\|A\|$. Then there exists $C$ large enough such that $f^{qc}_{C,T}(0)=0$ if and only if $T$
 is injective (and a similar framing can be obtained for the higher-order version of the theorem as well). This connection has been taken advantage of
 several times to obtain proofs of the theorem via calculus of variations:
 in particular Kirchheim and Kristensen (\cite[Theorem 1.3]{KirchheimJan})
 obtained a far reaching generalisation of Ornstein's result, and Faraco
 and Guerra (\cite{Faraco2022}) used a similar method to obtain
 a very efficient proof of Ornstein's result (for operators of order $1$ and $2$)
 in $\mathbb{R}^{2\times 2}.$ Both proofs, however, are qualitative,
 in that they do not provide witnesses to the failure of the inequality.
 In that direction, Conti, Faraco and Maggi constructed
 in \cite{Conti2005AFunctions}
 an explicit family of laminates which proves the failure of the inequality
 for Korn's operator $P_{\mathrm{Sym}}$. In this section, we extend their
 result by using our method to construct such a family for any operator
 $L$ which contains a matrix of rank $2$ in its kernel; moreover, we provide a quantitative bound on the failure of the inequality, showing that the
 necessary constant blows up like $p^*-1$. In a different quantitative
 direction, we mention the work \cite{Kazaniecki2023}, where Riesz products are used
 to construct functions that witness the failure of the $L^1$ inequality.
 \par To prove our result, we notice that
 the method applied in Theorem \ref{thm : sharp}
 of embedding scalar martingales can be 
 used in a much more general setting:
 \begin{theorem}
    Let $L$ be a homogeneous differential operator of order $1$ with 
 constant coefficients, i.e. $L(u)=T(\nabla u)$ for all $u\in C^\infty_c(\mathbb{R}^d, \mathbb{R}^m)$, with $T\in \mathrm{Lin}(M_{m,d}(\mathbb{R}), \mathbb{R}^k)$. If there exists a non-zero matrix of rank at most $2$
     in $\mathrm{ker}(T)$, then there is no
     constant $C$ such that for all $u\in C^\infty_c(\mathbb{R}^d,\mathbb{R}^m)$
     \begin{equation*}
         \|\nabla u\|_{L^1(\mathbb R^d)}\le C\|Lu\|_{L^1(\mathbb R^d)}.
     \end{equation*}
     Moreover, the minimal constant $C_p$ such that
     \begin{equation*}
         f_p(A):=C_p^p\|T(A)\|^p-\|A\|^p
     \end{equation*}
     is rank-one convex (which is finite if and only if $\mathrm{ker}(T)$ does not contain rank-one matrices) grows at least like $p^*-1$
     as $p\to 1$ or $p\to \infty$.
 \end{theorem}
 \begin{remark}
     If $d=2$, the theorem applies to all operators
     $T(\nabla u)$ where $T$ is not injective, proving
     the full Ornstein result. It also applies to $L=P_{\mathrm{Sym}},
     P_{\mathrm{Sym}_0}$.
     It is not difficult
     to see that the same method can be applied to the
     more general case of
     $\|T(\nabla u)\|_{L^1(\mathbb{R}^d)}\gtrsim \|S(\nabla u)\|_{L^1(\mathbb{R}^d)}.$
 \end{remark}
 \begin{proof}
 If $\mathrm{ker}(T)$ contains an element of rank $1$,
 the result is trivially true, so let us assume that
 there exists $A\in \mathrm{ker}(T)$ having rank $2$
 and that $\mathrm{ker}(T)\cap \mathcal R_1=\emptyset$.
 We can write $A=u\otimes v+w\otimes z$
 for some vectors $u,v,w,z$. Define $B:=u\otimes v-w\otimes z$. It is clear that $\mathrm{rank}(A-B)=
 \mathrm{rank}(A+B)=1$ and that $B\not \in\mathrm{ker}(T),$ for if it did then $ A-B$ would belong to $\mathrm{ker}(T)$,
 a contradiction. We can now apply the same martingale
 construction as in the proof of Theorem \ref{thm : sharp}: namely, let us construct
 first a (not necessarily orthogonal) projection $P$
 onto $\ker(T)$
 such that $P(B)=0$, and define the function
 \begin{equation*}
     f(A)=c_p^p\|T(A)\|^p-\|P(A)\|^p.
 \end{equation*}
 By considering the laminate associated with the
 martingale $M_n=g_nA+f_nB$, where $g_n$ is a $\pm1$ transform of $f_n$ and both start at $0$, the result follows. 
 \end{proof}
 \begin{remark}
     This proof can be seen as an extension of
     the laminate approach that Conti, Faraco and Maggi
     developed in \cite{Conti2005AFunctions}.
     Indeed, specialising the proof to $L=\mathrm{Sym}$
     leads to a family of laminates with the same support:
     the difference between the two constructions is then essentially one of choosing $f_n$.
 \end{remark}
If one is only interested in proving the failure
 of the inequality, then it is not necessary to utilise Burkholder's result or its sharpness: it suffices to construct
 two dyadic martingales $f_n, g_n$ such that $\mathrm{d}g_n=\varepsilon_n \mathrm{d}f_n$
 and $\sup \|f_n\|_{L^1(\mathbb P)}<\infty$ while $\|g_n\|_{L^1(\mathbb P)}\to \infty.$ Such an example can be found in \cite[Section 3.2.2]{osekowski}. We construct
 here an alternative example: 
 let $\Omega=[0,1]$, $\mathcal F_n$
 denote the dyadic filtration and construct $f_n$ via $f_0=0,$
 $\mathrm{d}f_n=2^{n-1}(2\chi_{I_n}-\chi_{I_{n-1}})$, where $I_n=[0,2^{-n}]$.
For $g_n$, we simply impose $\mathrm{d}g_n=(-1)^{n-1}\mathrm{d}f_n.$
 While the $L^1$
 norm of $f_n$ remains bounded (indeed one easily sees that $\|f_n\|_{L^1(\mathbb P)}\le 2$), $\|g_n\|_{L^1}$ grows 
 linearly: to see this let $T_k=I_k^c\cap I_{k-1}$
\begin{equation*}
    {g_n}_{|T_k}=\sum_{i=1}^{k-1} (-1)^{i-1}2^{i-1}+(-1)^k 2^{k-1}=(-1)^k\frac{2^{k+1}+(-1)^k}{3},
\end{equation*}
hence, since all the $T_k$ are disjoint, it follows that
\begin{equation*}
    \|g_n\|_{L^1}\ge \sum_{k\le n} \frac{2^{k+1}+(-1)^k}{3} 2^{-k}
    =\frac {2n}3+\mathcal O(1).
\end{equation*}
\begin{remark}
The following probabilistic interpretation of the above is possible:
consider a game played, at each turn, by tossing a coin.
The martingale $f_n$ denotes the payoff obtained by playing
$n$ turns with the following strategy:
betting each turn that the toss will result in heads,
doubling the bet each turn,
until the first tail is tossed,
at which point one stops playing.
$g_n$ is then the payoff obtained by a similar strategy,
where one doubles the bet at each turn and stops
once tails appears for the first time,
but instead of betting on heads each turn, alternates:
first heads, then tails, then heads etc.
\end{remark}
In \cite{Geiss} it was proved that for any homogeneous even
multiplier $T_m$ of order zero, one has
\begin{equation*}
\|T_m\|_{L^p(\mathbb{R}^d)\to L^p(\mathbb{R}^d)}\gtrsim p^*-1.
\end{equation*}
A similar result, under some additional
hypotheses, was proved for matrix-valued
multipliers in \cite[Theorem 14]{boros2011}.
These results, however, concern
themselves with the norm of $T^{-1}$ or, 
in other words, with 
$C^{\mathrm{qc}}(L,p)$, proving $C^{\mathrm{qc}}(L,p)
\gtrsim p^*-1$. Our result, on the other hand, deals with 
$p^*-1\lesssim C^{\mathrm{rc}}(L,p)\le C^{\mathrm{qc}}(L,p).$\par
Finally, we remark that it is possible to avoid the use
of the projection $P$ (instead considering $M_n$ in its entirety 
directly), and this approach
would probably allow for a (slightly) more accurate
estimate of $C^{\mathrm{rc}}(L,p)$. However, the 
complexity of the Burkholder side of the problem increases very 
significantly, as shown in \cite{Ivanisvili2015InequalityTransform}.
\section{Conclusion}
There are several interesting questions that we
have not managed to answer. 
The most natural one, finding
a closed form for $g_{\mathbb R}(d)$,
seems to be mainly topological in nature. We point out that even the
upper bounds are still getting improved,
see \cite{DominguezLam2021}.
Similarly, determining the constants associated
to the second form of Korn-type inequalities is
mostly a geometrical question,
and its answer is unknown even for the standard 
form of Korn's inequality.
Some questions that seem more approachable are:
\begin{enumerate}
    \item
    Is there an $X_d$ of dimension $= N(d)+\mathcal O(1)$
    such that 
    $C(X_d)\approx \sqrt d?$
    If not, what is the best possible rate of growth in terms of 
    $d$? Is it exponential,
    as suggested by the example $C(\mathcal H(d))?$
    \item What more can be said about $C(X,p)?$
    For example, under which conditions 
    is it true that $C(X,p)=C(X,2)(p^*-1)$?
    As we mentioned before, this problem 
    is connected to determining whether
    certain rank-one convex 
    functions are quasiconvex
    (and as such, it is an interesting question in and of 
    itself, even ignoring
    the applications that
    obtaining a sharp constant for the 
    inequalities might have).
    \item As we have seen, there
    exist subspaces $X$ of dimension 
    less than or equal to $\lceil \frac{d^2}2\rceil+1$ 
    for which $C(X,2)=1.$ A somewhat
    explicit description of these
    spaces would be interesting, 
    since at the moment the smallest subspace satisfying $C(X,2)=1$
    that can be explicitly described
    is, to our knowledge, $\mathrm{Sym}_0.$
\end{enumerate}
Let us finally mention a natural direction to explore:
extending the previous results to problems
concerning higher derivatives; we will return to this
question in future work.
\section*{Acknowledgements}
The author wishes to thank Prof. Petrovi\'{c}
for his comments and insights on the 
topic of bilinear nonsingular maps,
Prof. Kristensen for
the many insightful conversations and 
comments, and for pointing out 
Boman's papers (\cite{Boman1967PartialSpaces} in particular) to us and
 the fact that Theorem \ref{thm : solchipot} was already
known in the calculus of variations literature
and Prof. Spector for bringing \cite{DeFigueiredo} 
to our attention. The author is also grateful 
to the anonymous reviewers for their thoughtful comments.\par
The author acknowledges the financial support of 
the Mathematical Institute of the University of Oxford.\par
For the purpose of open access, the author has applied a CC BY public copyright licence to any author accepted manuscript arising from this submission.
\bibliographystyle{elsarticle-harv}   
\bibliography{minimal_korn}                 
\end{document}